\documentclass[12pt]{amsart}

\usepackage{graphicx}

\usepackage{subcaption} 

\usepackage{tikz-cd}
\usepackage{hyperref}
\hypersetup{bookmarks=true,
	unicode=true,
	colorlinks=true,
	citecolor=black,
	linkcolor=black,
	urlcolor=black,
	plainpages=false,
	pdfpagelabels=true}

 \usepackage{url}	
 \allowdisplaybreaks 

\usepackage{pgf}

\usepackage{xcolor}

\usepackage{comment} 

\usepackage{tikz}
\usetikzlibrary{arrows}
\usetikzlibrary{decorations.markings,arrows,automata,arrows,backgrounds,snakes}
\usepackage{tikz-cd} 
\usepackage{pgf}
\usetikzlibrary{babel}

\usepackage{appendix}

\usepackage{mathtools}

\usepackage{mathrsfs}
\usepackage{dsfont}
\usepackage{ulem}

\usepackage{mathtools}

\usepackage[shortlabels]{enumitem}

\numberwithin{equation}{section}

\newtheorem{theorem}{Theorem}[section]
\newtheorem{proposition}[theorem]{Proposition}
\newtheorem{lemma}[theorem]{Lemma}
\newtheorem{corollary}[theorem]{Corollary}

\newtheorem{conjetura}[theorem]{Conjecture}

\theoremstyle{definition}
\newtheorem{definition}[theorem]{Definition}
\newtheorem{example}[theorem]{Example}			
\newtheorem{idea}[theorem]{Idea}

\theoremstyle{remark}
\newtheorem{remark}[theorem]{Remark}
\newtheorem{notation}[theorem]{Notation}

\newtheorem{claim}[theorem]{Claim}

\usepackage{color}
\definecolor{darkgreen}{cmyk}{1,0,1,.2}
\definecolor{m}{rgb}{1,0.1,1}

\newdimen\theight
\def\TeXref#1{%
             \leavevmode\vadjust{\setbox0=\hbox{{\tt
                     \quad\quad  {\small \textrm #1}}}%
             \theight=\ht0
             \advance\theight by \lineskip
             \kern -\theight \vbox to
             \theight{\rightline{\rlap{\box0}}%
             \vss}%
             }}%

\begin{document}

 \title[New applications of the Lefschetz number]{New applications and computations of the Lefschetz number of homeomorphisms and open maps}
 \thanks{The authors were partially supported by the grants PID2020-114474GB-I00 (AEI/FEDER, UE) and ED431C 2023/31 (Xunta de Galicia, FEDER). The second author was partially supported by Programa de axudas á etapa predoutoral da Xunta de Galicia.
 }

\author[J. Álvarez \and A. Majadas-Moure 
     ]{%
	Jesús A. Álvarez López \and Alejandro O. Majadas-Moure 
}

\address{
         Jesús A. Álvarez López\\
         Departamento de Matemáticas, Universidade de Santiago de Compostela,
           Spain}
         \email{jesus.alvarez@usc.es}
              
\address{
		 Alejandro O. Majadas-Moure \\
		 Departamento de Matemáticas, Universidade de Santiago de Compostela, Spain}
		 \email{alejandro.majadas@usc.es}
		
		

\begin{abstract} 
We show that the combinatorial Lefschetz number is a topological invariant. This is an important result in itself; in order to point it out, we will also work here several relevant consequences in different directions. The first of them is a significant simplification of the computations involved in obtaining the Lefschetz number of certain maps, as well as some new Lefschetz fixed-point theorems for unbounded spaces. Indeed, these ideas allow us to obtain a clear lower bound for the Nielsen number of a triad in some spaces, such as, for example, the connected sum of two $p$-tori ($p>2$). Another consequence, in the case of homeomorphisms, is that, in the classical axiomatic definition of the Lefschetz number, the wedge-of-circles axiom and the cofibration axiom can be replaced by the single axiom of topological invariance of the combinatorial Lefschetz number. Using the invariance of the combinatorial Lefschetz number we generalize O'Neill's classical result about topological invariance of the fixed-point index. We also generalize the combinatorial Lefschetz number from homeomorphisms to open maps and we obtain a new fixed-point theorem.

\end{abstract}



\maketitle

\section{Introduction}
The Lefschetz number is the main topological invariant used in fixed-point theory. The classical Lefschetz fixed-point theorem (for self-maps of compact simplicial complexes) is one of the main tools one can use when trying to detect fixed points of a map. However, the computations are sometimes really tedious and, in many non-compact spaces (in particular, most of the unbounded spaces), the existence of a Lefschetz fixed-point theorem remains largely unknown to this day, being one of the most important open problems in topological fixed-point theory \cite[Problem 3]{Brown2}.

When $f$ is a homeomorphism, the Lefschetz fixed-point theorem was extended in \cite{M-M1}  by Mosquera-Lois and the second author to non-necessarily closed $f$-invariant definable spaces inside a simplicial complex. Indeed, this new number provides a better way of counting the number of fixed points of a homeomorphism, as \cite[Example 4.3]{M-M1} shows. However, even if some topological invariance was proved in that paper \cite[Remark 3.5]{M-M1}, it was not proved that the combinatorial Lefschetz number was a topological invariant in the usual sense. In \cite[Corollary 3.4]{M-M2} and \cite[Theorem 2.12]{A-M-M} Mosquera-Lois and the authors presented more partial results of the topological invariance on the combinatorial Lefschetz number.

In this paper, we prove the topological invariance of the combinatorial Lefschetz number, completing the progress of \cite{M-M1,M-M2,A-M-M}. This has several important consequences, some of which we explore here.

Moreover, we extend the combinatorial Lefschetz number from a homeomorphism $f$ and definable subspaces $A\subset X$ such that $f(A)= A$ to open maps $f$ and definable subspaces $A\subset X$ such that $f(A)\subset A$ and $f(X\setminus A)\subset X\setminus A$ (in fact, we only need $A$ to be homeomorphic to a definable space). Consequently, a new Lefschetz fixed-point theorem is obtained for bounded spaces that are not necessarily closed (Theorem~\ref{teor punto fijo}). This new theorem will imply Theorem~\ref{teor pto fijo de la extension}, about the existence of fixed points of the extensions of maps.

The topological invariance of the combinatorial Lefschetz number has many applications. A showcase of the relevance of this tool is that it can even be applied in Nielsen theory to bound some Nielsen numbers in certain spaces, such as connected sums of tori (Theorem~\ref{idea acotacion nielsen}, Corollary~\ref{corolario nielsen}). Moreover, when the maps are open (and sometimes even homotopic to open maps), the use of the combinatorial Lefschetz number and, especially, of its topological invariance, also becomes a very powerful tool to compute the Lefschetz number. This is quite important since  simplifications sometimes allow the computation Lefschetz numbers in an almost straightforward way. In fact, we will show that the topological invariance of the combinatorial Lefschetz number is equivalent to the cofibration and wedge-of-circles axioms in \cite{Arkowitz}. Indeed, this advantage of computing Lefschetz numbers with the combinatorial Lefschetz number can be exploited to obtain some techniques and fixed-point results in unbounded spaces (for example, Theorems~\ref{primer resultado gordo, wedge}, \ref{primer resultado gordo, grafos}, \ref{primer resultado gordo, super con borde} and \ref{primer resultado gordo, superficies}). Example~\ref{ej compactificaciones al reves} illustrates very well the relevance of the combinatorial Lefschetz number in these techniques.

Another consequence of this topological invariance is a generalization of O'Neill's classical result on the topological invariance of the fixed-point index \cite[Theorem 2.5]{ONeill}, which has been the main topological-invariance result for the index (not necessarily for maps with isolated fixed points) for more than $70$ years and which we generalize in Theorem~\ref{generalizacion oneil} for open maps. Furthermore, a topological-invariance result for the relative Lefschetz number (also known as the Lefschetz number of a pair of spaces \cite[Definition 4.1]{Gorniewicz}) is obtained in Corollary~\ref{inv numero relativo} as a direct consequence of the invariance of the combinatorial Lefschetz number.

In Section~\ref{inv top}, we prove the topological invariance of the combinatorial Lefschetz number. In this section, we also extend the definition of combinatorial Lefschetz number from homeomorphisms to open maps. In Remark~\ref{nota amplitud def}, we show that this new context is clearly more general than that of \cite{M-M1}. We also see that this is almost all we can relax the hypotheses on the maps when defining the combinatorial Lefschetz number (Example~\ref{ej hipotesis}). Moreover, a new fixed-point theorem is obtained in Theorem~\ref{teor punto fijo}.


In Example~\ref{ejemplo mejor contar} we show that, even when the map is only open and not a homeomorphism, the combinatorial Lefschetz number provides a better way to count the number of fixed points than the classical Lefschetz number.

Finally, in Theorem~\ref{generalizacion oneil}, we obtain a generalization (for open maps) of O'Neill's classical result about the topological invariance of the fixed-point index.


In Section~\ref{axiomas y ejemplos}, we start by proving that the topological invariance of the combinatorial Lefschetz number of homeomorphisms is equivalent to the cofibration and wedge-of-circles axioms. Note that this is quite relevant, since the axiomatic constructions of the invariants used in fixed-point theory are an important topic in the area due to the applications of these constructions (see, for example, \cite{G-S,Staecker1,Staecker2}, or even \cite{Brown,ONeill}, or Dold's classic works for more examples of these axiomatic constructions).

The computation of the Nielsen numbers constitutes the main challenge in Nielsen theory and also one of the main interests in topological fixed-point theory. The fact that our theorem on topological invariance of the combinatorial Lefschetz number can be applied to obtain in Section~\ref{seccion cotas nielsen} some new results in Nielsen theory adds evidence for the usefulness of this topological invariance. Among the few computations existing in the literature, in that section, we highlight \cite{B-B-P-T}, where the Nielsen number is computed on the tori, and \cite{Brown fibrados}, where, under quite restricted hypothesis, a product rule is given for the Nielsen number on fiber spaces. Also, in \cite{Kelly 2}, an algorithm is given for computing the Nielsen number of surface maps homotopic to a homeomorphism. Finally, the authors of \cite{D-H-T} consider the case of some particular maps between orientable surfaces, paying special attention to some specific self-maps of the connected sum of two tori, and presenting a technique that made it possible to obtain in this surface the Nielsen number of some of the maps for which it remained unknown. Another algorithm, this time for surfaces with boundary, is shown in \cite{Wagner}.

More precisely, in this Section~\ref{seccion cotas nielsen}, we explain a technique to obtain some lower bounds for the Nielsen number of a triad introduced in \cite{Schirmer}. However, the choice of this number instead of the classical Nielsen number is because it satisfies an additive property \cite[Theorem 4.12]{Schirmer} in the spaces we will deal with. If an additive behavior is found for the Nielsen number under some conditions, the same argument can be used to obtain new bounds.


In Section~\ref{subseccion de los ejemplos}, we present, through many examples, new techniques of computation of the Lefschetz number involving the combinatorial Lefschetz number and its topological invariance. As we will see, in many situations, this turns into a much better tool to compute the Lefschetz number, rather than the homological computation or the use of axioms in \cite{Arkowitz}. 

Finally, in Section~\ref{seccion aplicaciones practicas}, we begin by defining the combinatorial Lefschetz number for unbounded spaces. As a direct consequence, we obtain a result about the existence of new fixed points in the extensions of fixed-point-free maps (this topic is related with Nielsen theory of the complement and, more in particular, with many works as, for example, \cite{K-B-L-D2,K-B-L-D,Schirmer2,Zhao2}). Next, the use of the combinatorial Lefschetz number instead of the classical Lefschetz number will allow us to state and prove new fixed-point results for some unbounded spaces (Theorems~\ref{primer resultado gordo, wedge}, \ref{primer resultado gordo, grafos}, \ref{primer resultado gordo, super con borde} and \ref{primer resultado gordo, superficies}). Such results would not be true for the classical Lefschetz number. In this section we also explain with some examples how these results become significantly relevant thanks to the topological invariance of the combinatorial Lefschetz number.

Starting with \cite{Leray}, several works have been made in order to obtain Lefschetz fixed-point theorems on unbounded spaces; for example \cite{Browder2,E-F,Granas,Nussbaum1,Nussbaum2,Nussbaum3,Tromba}. The case of infinite-dimensional spaces has also been explored \cite{Browder1}. Nevertheless, little progress has been made recently in this direction, except for \cite{Cauty2,Cauty}. Indeed, compactness was always present in these results (which were stated for compact maps, CAC, condensing maps, ...). In contrast, the tools used to obtain our new results for unbounded spaces are the bounds of the index. Hence, since the only compactness restriction now is to consider proper maps, the results in Section~\ref{seccion aplicaciones practicas} can be regarded as the first work giving a Lefschetz fixed-point theorem on unbounded spaces without using compactness (in \cite{Hochs}, compactness is not used, but the maps are required to be isometries, which are much more restrictive than open maps or even homeomorphisms).

Finally, in Section~\ref{subseccion nuevas acotaciones}, we explain how the same ideas that we use to obtain fixed-point results for unbounded spaces can be used to find counterexamples of spaces where the index is not bounded.

In this paper, many examples will be presented. Each of them is focused on a specific advantage of the combinatorial Lefschetz number. Sometimes, the examples may look similar, but each of them presents nuances that make them important in the paper.


\subsection*{Acknowledgements} The second author thanks Michael Kelly for enlightening discussions. In particular, his help has been very useful for understanding the bounds of the index. He also thanks Chris Staecker for discussions concerning the Nielsen number.
\section{Preliminaires}

For a rigorous definition of the Lefschetz number, we recommend \cite{M-M1}. For a main idea and basic notions of $o$-minimal structures and definable spaces, we recommend \cite[Sec.~2]{M-M2}.
We recall here the main results that will be used. Unless otherwise stated, our simplicial complexes will be finite. We will write $\varLambda(f,X)$ for the usual Lefschetz number of a map $f$ on a simplicial complex $X$ and $\varLambda(f,U)_X$ for the combinatorial Lefschetz number of $f$ in $U$ relative to $X$ \cite[Definition 3.4]{M-M1}.

The Lefschetz combinatorial number, almost by definition, is additive in the following way:

\begin{proposition}[{Inclusion-exclusion principle \cite{M-M1}}]\label{inclusion exclusion}
    Let $X$ be a simplicial complex, $f:X\rightarrow X$ a homeomorphism, and $U$ and $V$ definable $f$-invariant {\rm(}$f(U)=U$ and $f(V)=V${\rm)} subsets of $X$. Then
    \begin{equation*}
        \varLambda (f,U\cup V)_X=\varLambda(f,U)_X+\varLambda(f,V)_X-\varLambda(f,U\cap V)_X\;.
    \end{equation*}
\end{proposition}
\begin{remark}\label{relacion con relativo}
In particular, given a simplicial complex $X$, a homeomorphism $f:X\rightarrow X$ and a definable $f$-invariant subspace $A\subset X$, we have
\begin{equation*}
    \varLambda(f,X)=\varLambda(f,X)_X=\varLambda(f,A)_X+\varLambda(f,X\setminus A)_X\;,
\end{equation*}
where the first equality follows from the definition of the combinatorial Lefschetz number \cite{M-M1}, since $X$ is compact. But this equation is the same as
\begin{equation*}
    \varLambda(f,A)_X=\varLambda(f,X)-\varLambda(f,X\setminus A)_X\;.
\end{equation*}
Now, if $A$ is open in $X$, we have that $X\setminus A$ is, under certain triangulation (see Theorem~\ref{thm:triangulation}), a subcomplex of $X$, and so $\varLambda(f,X\setminus A)_X=\varLambda(f_{|X\setminus A},X\setminus A)$.
But, from \cite[Property 4.4]{Gorniewicz}, we know that, if $\varLambda(f,X)$ and $\varLambda(f_{|X\setminus A},X\setminus A)$ are defined, then
\begin{equation*}
    \varLambda(f,X)-\varLambda(f_{|X\setminus A},X\setminus A)=\varLambda(f,(X,X\setminus A))\;,
\end{equation*}
where $\varLambda(f,(X,X\setminus A))$ denotes the relative Lefschetz number \cite[Definition 4.1]{Gorniewicz}. As a consequence, if $A$ is open, we have $\varLambda(f,A)_X=\varLambda(f,(X,X\setminus A))$. So the combinatorial Lefschetz number is a generalization of the relative Lefschetz number since, in order to consider $\varLambda(f,A)_X$, $A$ does not need to be open.
\end{remark}

We may assume that our o-minimal structures contain the semilinear sets. In this case:

\begin{theorem}[{Definable-triangulation theorem \cite{Dries}}] \label{thm:triangulation} 
	Let $X\subset \mathbb{R}^n$ be a definable set and let $\{X_i\}_{i=1}^{m}$ be a finite family of definable subsets of $X$. Then there exists a definable triangulation of $X$ compatible with the collection of subsets.
\end{theorem}
\begin{remark}\label{obs vale para homeomorfos a definibles}
    Even if all results in \cite{M-M1} and in the current paper are stated for definable spaces, they can be carefully generalized for spaces that are homeomorphic to such definable spaces. For example, given a compact space $L\subset\mathbb{R}^n$, a homeomorphisms $g:L\to L$ and a $g$-invariant subspace $A\subset L$, if $(L,A,g)$ can be triangulated by $(X,U,f)$, where $X$ is a simplicial complex and $U$ is a definable $f$-invariant subspace for a homeomorphism $f:X\rightarrow X$, we can define $\varLambda(g,A)_L$ as $\varLambda(f,U)_X$. This is well defined. The proof is analogous to that of \cite[Lemma 3.11]{M-M1}. This is a broad generalization since the family of spaces we can work with is much larger than the definable ones. However, since this paper focuses on highlighting the relevance of the topological invariance of the combinatorial Lefschetz number through various applications where different tools are needed, by simplicity in the statements, we will continue to write the results in terms of definable spaces.
\end{remark}

\section{Topological invariance of the combinatorial Lefschetz number}\label{inv top}

From the definition of the combinatorial Lefschetz number, it can easily be obtained that, given homeomorphisms $f:X\rightarrow X$ and $g:Y\rightarrow Y$ and $f$- and $g$-invariant definable sets $A\subset X$ and $B\subset Y$ such that there exists a homeomorphism $h:\overline{A}\rightarrow \overline{B}$ so that $h(A)=B$ and
\[\begin{tikzcd}
	{\overline{A}} & {\overline{A}} \\
	{\overline{B}} & {\overline{B}}
	\arrow["f", from=1-1, to=1-2]
	\arrow["h"', from=1-1, to=2-1]
	\arrow["h"', from=1-2, to=2-2]
	\arrow["g", from=2-1, to=2-2]
\end{tikzcd}\]
commutes, we have $\varLambda(f,A)_X=\varLambda(g,B)_Y$.

The problem is that, when $h:\overline A\to\overline B$ is not given, even if $A$ and $B$ are homeomorphic, their closures need not to be homeomorphic. 

In address these possibilities, some partial results are obtained in \cite[Corollary 3.4]{M-M2} and in \cite[Theorem 2.12]{A-M-M}. The one in \cite[Corollary 3.4]{M-M2} is more categorical and direct, but it is based on some restrictive hypothesis (for example the homeomorphism must be cellular). However, in \cite[Theorem 2.12]{A-M-M}, a result is obtained when $f$ has no fixed points in $\overline{A}\setminus\mathring{A}$. The idea is that, in this case, the combinatorial Lefschetz number equals the index of $f$ at the interior of $A$ and, under certain hypothesis, it is possible to apply the topological invariance of the index given in \cite[Theorem 2.5]{ONeill}. In fact, this relation between the index and the combinatorial Lefschetz number is what will allow us to extend \cite[Theorem 2.5]{ONeill} in Section~\ref{inv top}, once we have proved the topological invariance of the combinatorial Lefschetz number in a completely different way.

To prove the topological invariance, we need a preliminary lemma. All over this section, $\mathring{A}$ will always denote the interior of $A$ in $\overline{A}$.

\begin{lemma}\label{lema principal}
    Let $A\subset X$ and $B\subset Y$ be definable sets of simplical complexes $X$ and $Y$ such that there exists a homeomorphism $h:A\rightarrow B$. Consider the closures $\overline{A}$ and $\overline{B}$ of $A$ and $B$ in $X$ and $Y$. Then, $h$ maps $\mathring{A}$ to $\mathring{B}$ and hence {\rm(}$f$ is a homemorphism{\rm)} $A\setminus \mathring{A}$ to $B\setminus\mathring{B}$, where $\mathring{A}$ {\rm(}resp., $\mathring{B}${\rm)} denotes the interior of $A$ in $\overline{A}$ {\rm(}resp., $\overline{B}${\rm)}.
\end{lemma}
\begin{proof}
    Let $b\in B\setminus\mathring{B}$ and $a\in\mathring{A}$ such that $h(a)=b$. Consider open neighborhoods $U$ of $a$ in $A$ and $V$ of $b$ in $B$, such that $\overline{U}\subset\mathring{A}$ and $h$ maps $U$ homeomorphically to $V$. Since $V$ is open in the relative topology of $B$, there exist $V'$ open in $\overline{B}$ such that $V'\cap B=V$. Then, since $b\in B\setminus\mathring{B}$, there exists some $b'\in\overline{B}\setminus B$ contained in $V'$, and since $b'\in\overline{B}$, there must exist a sequence $\{b_n\}_{n\in\mathbb{N}}$ contained in $V'\cap B=V$ which converges to $b'$. Consider now the sequence $\{a_n\}_{n\in\mathbb{N}}$ consisting of the inverse images of $\{b_n\}$ by $f$. We have that $\{a_n\}$ is contained in $U$ and, since $\overline{U}\subset\mathring{A}$ is bounded and hence compact, there must exist a subsequence $\{a_{n_k}\}_{k\in\mathbb{N}}$ which converges to a point $a'\in A$. Then, by continuity, $\{f(a_{n_k})\}=\{b_{n_k}\}$ must converge to $f(a')\in B$. But $\{b_n\}$ (and hence also $\{b_{n_k}\}$) converges to $b'\notin B$.\qedhere
\end{proof}

Now we can prove the topological invariance.

\begin{theorem}\label{teor inv topo}
    Let $A\subset X$ and $B\subset Y$ be definable sets of simplical complexes $X$ and $Y$ such that there exists homeomorphisms $f:X\rightarrow X$ and $g:Y\rightarrow Y$ which restrict to homeomorphisms $f_{|A}:A\rightarrow A$ and $g_{|B}:B\rightarrow B$ and such that there is a homeomorphism $h:A\rightarrow B$ making the diagram
    \[\begin{tikzcd}
	A & B \\
	A & B
	\arrow["h", from=1-1, to=1-2]
	\arrow["f"', from=1-1, to=2-1]
	\arrow["g"', from=1-2, to=2-2]
	\arrow["h"', from=2-1, to=2-2]
\end{tikzcd}\]
commutative. Then $\varLambda(f,A)_X=\varLambda(g,B)_Y$, where $\varLambda(f,A)_X$ denotes the combinatorial Lefschetz number of $f$ in $A$ with respect to $X$.
\end{theorem}
\begin{proof}
Due to the definable-triangulation theorem of definable sets \cite{Dries}, we can assume that $A$ and $B$ are generalized simplicial complexes (this means a union of open faces of the simplex).
We proceed by induction on the dimensions of $A$ and $B$. If $A$ and $B$ are of dimension $0$, each one is a finite set of points and the result follows.

    Since $\varLambda(f,A)_X=\varLambda(f,A)_{\overline{A}}$ and $\varLambda(g,B)_Y=\varLambda(g,B)_{\overline{B}}$ \cite[Theorem~3.9]{M-M1}, we may assume $X=\overline{A}$ and $Y=\overline{B}$. To simplify, we will write $f$ instead of $f_{|\overline{A}}$.

    We have descompositions
    \begin{align*}
        &\overline{A}=\mathring{A}\sqcup (\overline{A}\setminus A) \sqcup (\overline{\overline{A}\setminus A}\cap A)\\
        &\overline{B}=\mathring{B}\sqcup (\overline{B}\setminus B) \sqcup (\overline{\overline{B}\setminus B}\cap B)\;.
    \end{align*}  
 (Remember that $\mathring{A}$ denotes the interior of $A$ as subspace of $\overline{A}$). Indeed, since $A$ (resp., $B$) is $f$-invariant (resp., $g$-invariant) and definable, and $\overline{A}$ and $\overline{B}$ are definable (\cite[Lemma 1.3.4]{Dries} shows that the interior and closure of definable sets are definable), we have that $\overline{A}\setminus A$ (resp., $\overline{B}\setminus B$), $\mathring{A}$ (resp., $\mathring{B}$) , $\overline{\overline{A}\setminus A}$ (resp., $\overline{\overline{B}\setminus B}$) and $\overline{\overline{A}\setminus A}\cap A$ (resp., $\overline{\overline{B}\setminus B}\cap B$) are $f$-invariant (resp., $g$-invariant) and definable.
Then, by the additivity of the combinatorial Lefschetz number, 
\begin{multline}\label{primera identidad}
    \varLambda(f,\overline{A})=\varLambda(f,\overline{A})_{\overline{A}}\\=\varLambda(f,\mathring{A})_{\overline{A}}+\varLambda(f,\overline{A}\setminus A)_{\overline{A}}+\varLambda(f,\overline{\overline{A}\setminus A}\cap A)_{\overline{A}}\;,
\end{multline}
On the other hand, due to the cofibration axiom of the Lefschetz number \cite[Theorem 1.1]{Arkowitz}, we have
\begin{equation}\label{segunda identidad}
    \varLambda(f,\overline{A})=\varLambda(f_{\overline{\overline{A}\setminus A}},\overline{\overline{A}\setminus A})+\varLambda(\tilde{f},\overline{A}/(\overline{\overline{A}\setminus A}))-1\;,
\end{equation}
where $\tilde{f}$ is the map induced by $f$ in the quotient. 
Due to the additivity of the combinatorial Lefschetz number and to \cite[Theorem 3.9]{M-M1}, we have
\begin{multline}\label{tercera identidad}
    \varLambda(f_{|\overline{\overline{A}\setminus A}},\overline{\overline{A}\setminus A})=\varLambda(f_{|\overline{\overline{A}\setminus A}},\overline{\overline{A}\setminus A})_{\overline{\overline{A}\setminus A}}=\varLambda(f,\overline{\overline{A}\setminus A})_{\overline{A}}\\=\varLambda(f,\overline{A}\setminus A)_{\overline{A}} + \varLambda(f,\overline{\overline{A}\setminus A}\cap A)_{\overline{A}}\;.
\end{multline}

Now, combining Equations \eqref{primera identidad}, \eqref{segunda identidad} and \eqref{tercera identidad}, we obtain
\begin{equation}\label{cuarta identidad}
    \varLambda(f,\mathring{A})_{\overline{A}}=\varLambda(\tilde{f},\overline{A}/(\overline{\overline{A}\setminus A}))-1\;.
\end{equation}
In a similar way, we obtain
\begin{equation}\label{quinta identidad}
    \varLambda(g,\mathring{B})_{\overline{B}}=\varLambda(\tilde{g},\overline{B}/(\overline{\overline{B}\setminus B}))-1\;.
\end{equation}
Since
\begin{equation*}
    \varLambda(f,A)_{\overline{A}}=\varLambda(f,\mathring{A})_{\overline{A}}+\varLambda(f,\overline{\overline{A}\setminus A}\cap A)_{\overline{A}}\;,
\end{equation*}
\begin{equation*}
    \varLambda(g,B)_{\overline{B}}=\varLambda(g,\mathring{B})_{\overline{B}}+\varLambda(g,\overline{\overline{B}\setminus B}\cap B)_{\overline{B}}\;,
\end{equation*}
we will prove 
\begin{equation}\label{ecuacion accesoria}
\varLambda(f,\mathring{A})_{\overline{A}}=\varLambda(g,\mathring{B})_{\overline{B}}\;,\quad\varLambda(f,\overline{\overline{A}\setminus A}\cap A)_{\overline{A}}=\varLambda(g,\overline{\overline{B}\setminus B}\cap B)_{\overline{B}}\;.
\end{equation}

To prove the last equality of \eqref{ecuacion accesoria}, it suffices to show that $h$ maps $\overline{\overline{A}\setminus A}\cap A$ homeomorphically to $\overline{\overline{B}\setminus B}\cap B$, since, as  $\mathrm{dim}\bigl(\overline{\overline{A}\setminus A}\cap A\bigr)<\mathrm{dim}(A)$, we apply induction. But this follows from Lemma~\ref{lema principal}.

Using Equations \eqref{cuarta identidad} and \eqref{quinta identidad}, to show the first equality of \eqref{ecuacion accesoria}, it is enough to show that
\begin{equation*}
    \varLambda(\tilde{f},\overline{A}/(\overline{\overline{A}\setminus A}))=\varLambda(\tilde{g},\overline{B}/(\overline{\overline{B}\setminus B}).
\end{equation*}
But this follows since both spaces are homeomorphic (using again Lemma~\ref{lema principal}).
\end{proof}

So we have proved the following:

\begin{theorem}
    The combinatorial Lefschetz number is a topological invariant.
\end{theorem}
From Remark~\ref{relacion con relativo}, we get the following corollary.
\begin{corollary}\label{inv numero relativo}
    Let $X$ be a simplicial complex and $C\subset X$ a subcomplex. Let $f:X\rightarrow X$ be a homeomorphism such that $f(C)=C$ {\rm(}using Remark~\ref{obs numero para abtas es inv} we can extend this result to open maps such that $f(C)\subset C$ and $f(X\setminus C)\subset X\setminus C${\rm)}. Consider now another simplicial complex $Y$ with a subcomplex $D$ and a homeomorphism $g:Y\rightarrow Y$ such that $g(D)=D$. If there exists a homeomorphism $h:X\setminus C\rightarrow Y\setminus D$ such that $g_{|Y\setminus D}\circ h=h\circ f_{|X\setminus C}$, then 
    \begin{equation*}
        \varLambda(f,(X,C))=\varLambda(g,(Y,D))\;.
    \end{equation*}
\end{corollary}
\begin{proof}
    From Remark~\ref{relacion con relativo}, we know that $\varLambda(f,(X,C))=\varLambda(f,X\setminus C)_X$ and $\varLambda(g,(Y,D))=\varLambda(g,Y\setminus D)_Y$. But, by Theorem~\ref{teor inv topo}, we have 
    \begin{equation*}
    \varLambda(f,X\setminus C)_X=\varLambda(g,Y\setminus D)_Y\;.\qedhere
    \end{equation*}
\end{proof}

Now we want to explore when the requirement that $f$ be homeomorphism is necessary to have a good definition of combinatorial Lefschetz number.

The idea of the construction of the combinatorial Lefschetz number is that, given $A\subset X$ and a map $f:X\rightarrow X$, we consider a simplicial approximation $f^{\mathrm{simp}}$ and restrict the matrices of the automorphism induced by $f^{\mathrm{simp}}$ in  simplicial chains to the faces whose interior belongs to $A$ when computing the traces (again we refer to \cite[Sec.~2]{M-M2} for a sketch of this definition). However, when $f$ is not a homeomorphism, this idea is not well defined in general.
\begin{example}\label{ej hipotesis}
    Consider $X=[0,1]$, $A=(0,1)$ and $f:X\rightarrow X$ given by $f(x)=\frac{1}{2}$. The simplicial map $f^{\mathrm{s}}$ that sends $0$ to $0$ and $1$ to $0$ is a simplicial approximation of $f$. The computation of the combinatorial Lefschetz number of $f$ in $A$ with this approximation gives $0$. However, the identity on $X$ is another simplicial approximation of $f$, and with this approximation the computation of the combinatorial Lefschetz number of $f$ in $A$ gives $-1$.
\end{example}
Nevertheless, under certain hypothesis less restrictive than homeomorphisms, we can define the combinatorial Lefschetz number. The definition is exactly analogous to that of \cite{M-M1}. That is, with the notation in \cite[Sec.~3.2]{M-M1}, given a simplicial complex $Z$, a definable subspace $D\subset Z$ and a map $l:Z\rightarrow Z$, first we take a triangulation $(X,A,f)$ of $(Z,D,l)$, and a simplicial approximation $f^{\mathrm{simp}}$ of $f$. Then we define
\begin{equation*}
    \varLambda(l,D)_Z\vcentcolon =\varLambda_c(f,A)\vcentcolon=\varLambda^c(f^{\mathrm{simp}},A)\vcentcolon=\sum_p (-1)^p \mathrm{tr}(M_p(f^{\mathrm{simp}},X)_{|A})\;,
\end{equation*}
where $\varLambda(l,D)_Z$ denotes the combinatorial Lefschetz number. We will see that this is well defined sometimes, even if $l$ is not a homeomorphism.
\begin{theorem}\label{enunciado corto buena definicion}
    The combinatorial Lefschetz number is well defined for open maps $l:Z\rightarrow Z$ of a simplicial complex and for definable subsets $D$ such that $f(D)\subset D$ and $f(Z\setminus D)\subset Z\setminus D$.
\end{theorem}
\begin{remark}
    In addition to the hypothesis that $l$ is homeomorphism, in \cite{M-M1} $D$ had to be $f$-invariant ($f(D)=D$).
\end{remark}
In order to prove Theorem~\ref{enunciado corto buena definicion}, we need to prove that the definition of $\varLambda(l,D)_Z$ does not depend on the triangulation or on the simplicial approximations. So Theorem~\ref{enunciado corto buena definicion} can be reformulated as:

\begin{theorem}\label{teor donde demuestro la buena def para abtas}
    Let $D\subset Z$ be a definable set included in a simplicial complex $Z$ and $l:Z\rightarrow Z$ an open map that sends $D$ to $D$ and $Z\setminus D$ to $Z\setminus D$. Consider two triangulations $(X,A,f)$, $(Y,B,g)$ of $(Z,D,h)$ given by the triangulation theorem \cite{Dries}. Then, with the notation of \cite{M-M1},
    \begin{equation*}
        \varLambda_c(f,A)=\varLambda_c(g,B)\;.
    \end{equation*}
\end{theorem}
\begin{proof}
    Recall that $\varLambda_c(f,A)=\varLambda^c(f^{\mathrm{simp}},A)$ and $\varLambda_c(g,B)=\varLambda^c(g^{\mathrm{simp}},B)$ for simplicial approximations $f^{\mathrm{simp}}$ and $g^{\mathrm{simp}}$ of $f$ and $g$ (one case can be $A=B$ and $f^{\mathrm{simp}}$ and $g^{\mathrm{simp}}$ different simplicial approximations of $f$, so we only have to prove the independence of the definition from the triangulation).

    We use induction on the dimensions of $A$ and $B$. If both of them have dimension $0$, then they consist of the same finite number of points and then 
    \begin{equation*}
        \varLambda^c(f^{\mathrm{simp}},A)=\varLambda(f,A)=\varLambda(g,B)=\varLambda^c(g^{\mathrm{simp}},B)\;.
    \end{equation*}

    In general, since $f(A)\subset A$, and $f(X\setminus A)\subset X\setminus A$, we have $f(\overline{A})\subset \overline{A}$ and $f(\overline{X\setminus A})\subset\overline{X\setminus A}$. As a consequence, $f(\overline{A}\setminus A)\subset \overline{A}\setminus A$, and so also $f(\overline{\overline{A}\setminus A)}\subset \overline{\overline{A}\setminus A}$. But this means that $f$ maps  $\overline{A}\setminus\mathring{A}$ to $\overline{A}\setminus\mathring{A}$ (recall that $\mathring{A}$ denotes the interior of $A$ in $\overline{A}$).

    Indeed, since $f$ is open, we have $f(\mathring{A})\subset \mathring{A}$. To see this, first note that $f^{-1}(\overline{A})=\overline{A}$. This is true because, if $f$ maps some $x\in X\setminus \overline{A}$ to $\overline{A}$, then, since $f$ is open, it maps an open neighborhood $V$ of $x$ in $X\setminus \overline{A}$ to an open neighborhood of $f(x)$ in $\overline{A}$, mapping in particular a point of $X\setminus A$ to $A$.
    
    Second, since $f^{-1}(\overline{A})=\overline{A}$, we have that $f_{\overline{A}}$ is open. This is true because, if $U$ is open in $\overline{A}$, we have $U=U'\cap \overline{A}$ with $U'$ open in $X$, and then $f_{\overline{A}}(U)=f(U')\cap \overline{A}$. So $f_{\overline{A}}(U)$ is open in $\overline{A}$.

    Now, to show $f(\mathring{A})\subset \mathring{A}$, suppose that $f$ maps some point $a\in\mathring{A}$ to $\overline{A}\setminus \mathring{A}=\overline{\overline{A}\setminus A}$ (recall that $\mathring{A}$ denotes the interior of $A$ in $\overline{A}$). Take an open neighborhood $V$ of $a$ in $\mathring{A}$. Since $f_{\overline{A}}$ is open, it maps $V$ to an open neighborhood of $f(a)\in\overline{\overline{A}\setminus A}$, and hence $f$ maps some points of $\mathring{A}$ (and, in particular, of $A$) to points that are not in $A$, which is not possible since, by hypothesis, $f(A)\subset A$.

    So $f(\mathring{A})\subset \mathring{A}$. Analogously, we have $g(\mathring{B})\subset\mathring{B}$. We have also that $f(X\setminus\mathring{A})\subset X\setminus\mathring{A}$, since, if $x\notin A$, then $f(x)\notin A \supset\mathring{A}$, and, if $x\in A\setminus\mathring{A}$, then 
    the previous argument shows that, if $f(x)\in\mathring{A}$, there is an element of $\overline{A}\setminus A\subset X\setminus A$ mapped by $f$ to $A$, which is not possible. Similarly, we obtain that  $f$ maps $\overline{A}\setminus A$, $\overline{\overline{A}\setminus A}$ and $\overline{\overline{A}\setminus A}\cap A$ to themselves, and it also maps their complements to themselves. Indeed, they are all definable (they consist of interiors, closures and complements of definable spaces). Consequently, it is possible to consider the combinatorial Lefschetz number of $f$ on these spaces (since $(f^{\mathrm{simp}})_{|\overline{A}}$ is a simplicial approximation of $f_{|\overline{A}}$, we will still write $f$ or $f^{\mathrm{simp}}$ instead of its restrictions). By the additivity that follows from the definition of $\varLambda_c$, we have
    \begin{multline*}
        \varLambda(f,\overline{A})=\varLambda^c(f^{\mathrm{simp}},\overline{A})=\\\varLambda^c(f^{\mathrm{simp}},\mathring{A})+\varLambda^c(f^{\mathrm{simp}},\overline{\overline{A}\setminus A}\cap A)+\varLambda^c(f^{\mathrm{simp}},\overline{A}\setminus A)\\=\varLambda_c(f,\mathring{A})+\varLambda_c(f,\overline{\overline{A}\setminus A}\cap A)+\varLambda_c(f,\overline{A}\setminus A)\;,
    \end{multline*}
    where the first equality follows from Hopf's trace theorem. Since, $\overline{\overline{A}\setminus A}$ is a subcomplex, the Hopf's trace theorem also implies
    \begin{multline*}
        \varLambda_c(f,\overline{\overline{A}\setminus A}\cap A)+\varLambda_c(f,\overline{A}\setminus A)\\=\varLambda^c(f^{\mathrm{simp}},\overline{\overline{A}\setminus A}\cap A)+\varLambda^c(f^{\mathrm{simp}},\overline{A}\setminus A)\\=\varLambda^c(f^{\mathrm{simp}},\overline{\overline{A}\setminus A})=\varLambda(f,\overline{\overline{A}\setminus A})\;.
    \end{multline*}
    On the other hand, due to the cofibration axiom \cite[Theorem 1.1]{Arkowitz}, we have
    \begin{equation*}
        \varLambda(f,\overline{A})=\varLambda(f,\overline{\overline{A}\setminus A})+\varLambda(\tilde{f},\overline{A}/(\overline{\overline{A}\setminus A}))-1\;,
    \end{equation*}
    where $\tilde{f}$ denotes the map induced by $f$ on the quotient.

    Combining these expressions, we obtain
    \begin{equation*}
        \varLambda_c(f,\mathring{A})=\varLambda(\tilde{f},\overline{A}/(\overline{\overline{A}\setminus A}))-1\;.
    \end{equation*}
    Analogously, we get
    \begin{equation*}
        \varLambda_c(g,\mathring{B})=\varLambda(\tilde{g},\overline{B}/(\overline{\overline{B}\setminus B}))-1\;.
    \end{equation*}
    Then, by Lemma~\ref{lema principal}, we obtain $\varLambda(\tilde{f},\overline{A}/(\overline{\overline{A}\setminus A}))=\varLambda(\tilde{g},\overline{B}/(\overline{\overline{B}\setminus B}))$, and hence $\varLambda_c(f,\mathring{A})=\varLambda_c(g,\mathring{B})$.

    Finally, from
    \begin{align*}
    &\varLambda_c(f,A)=\varLambda_c(f,\mathring{A})+\varLambda_c(f,\overline{\overline{A}\setminus A}\cap A)\;,\\
    &\varLambda_c(g,B)=\varLambda_c(g,\mathring{B})+\varLambda_c(g,\overline{\overline{B}\setminus B}\cap B)\;,
    \end{align*}
    and using the induction hyphotesis ($\mathrm{dim}\bigl(\overline{\overline{A}\setminus A}\cap A\bigr)<\mathrm{dim}(A)$), we get $\varLambda_c(f,\overline{\overline{A}\setminus A}\cap A)=\varLambda_c(g,\overline{\overline{B}\setminus B}\cap B)$, and hence
    \begin{equation*}
        \varLambda_c(f,A)=\varLambda_c(g,B)\;.\qedhere
    \end{equation*}
\end{proof}
\begin{remark}\label{obs numero para abtas es inv}
    Again, the combinatorial Lefschetz number is a topological invariant in this case. The proof is analogous to the proof of the invariance when the maps are homeomorphisms (Theorem~\ref{teor inv topo}). It is also additive in the sense of Proposition~\ref{inclusion exclusion}. The argument is analogous to the argument given in \cite{M-M1}. Note that now the integration results of \cite{M-M1} can also be generalized to open maps.
\end{remark}
\begin{remark}\label{nota amplitud def}
    Note that the new hypothesis over  the map and the definable spaces are significantly weaker than those in \cite{M-M1}. For example, on the unit disk $D^2$ in $\mathbb{C}$ centered at $0$, the map that sends $0$ to $0$ and $z$ to $|z|\frac{z^p}{|z^p|}$ is open but far from being a homeomorphism.
\end{remark}
We have the following fixed-point theorem that generalizes \cite[Theorem~4.1]{M-M1}. Using Remark~\ref{relacion con relativo}, we see that it also generalizes \cite[Theorem~3.1]{Bowszyc} when the map is open. The proof is analogous to \cite[Theorem ~4.1]{M-M1}.

\begin{theorem}\label{teor punto fijo}
    Let $X$ be a simplicial complex, $f:X\rightarrow X$ an open map and $A\subset X$ a definable subspace such that $f(A)\subset A$ and $f(X\setminus A)\subset X\setminus A$. Then, if $\varLambda(f,A)_X\neq 0$, $f$ has a fixed point in $\overline{A}$.
\end{theorem}
    As is noted in \cite[Example 4.3]{M-M1}, sometimes the combinatorial Lefschetz number gives a better approach to the number of fixed points than the usual Lefschetz number. This also happens when $f$ is not a homeomorphism, but an open map.
\begin{example}\label{ejemplo mejor contar}
    Let $X$ be the annulus centered at $(0,0)\in\mathbb{R}^2$ of Figure~\ref{corona circular tesis}. The ray from $(0,0)$ and through each $x\in X$ meets the boundary of the annulus in two points. Let us denote by $p_x$ and $q_x$ these points (see Figure~\ref{corona circular tesis}). For each $x\in X$, there exists some $t_x\in[0,1]$ such that $x=t_x\cdot p_x+(1-t_x)\cdot q_x$. Let $f:X\rightarrow X$ be the open map defined by
    $$
f(x)=
\begin{cases}
(1-3t_x)\cdot p_x+3t_x\cdot q_x,\;t_x\in[0,1/3]\\
(3t_x-1)\cdot p_x+(2-3t_x)\cdot q_x,\;t_x\in [1/3,2/3]\\
(3-3t_x)\cdot p_x +(3t_x-2)\cdot q_x,\; t_x\in [2/3,1]
\end{cases}
$$
    The set of fixed points of $f$ consists of three circles. However, since $f$ is homotopic to the identity, we obtain $\varLambda(f,X)=\chi (X)=0$ and, however, even in a situation like this, where the fixed points are so obvious, the classic Lefschetz fixed-point theorem doesn't guarantee the existence of any fixed point.
    
    Now, let $X_1$ be the set consisting of $X$ without an open segment $(p_x,q_x)$ (for some $x\in X$). In this case, the additivity of the combinatorial Lefschetz number implies 
    \begin{equation}\label{ec del ejemplo}
        0=\varLambda(f,X)=\varLambda(f,X_1)_X+\varLambda(f,(p_x,q_x))_X\;.
    \end{equation}
    But, by the contractibility of the segment we have
    \begin{multline*}
    \varLambda(f,(p_x,q_x))_X+\varLambda(f,\{p_x,q_x\})_X=\\\varLambda(f,[p_x,q_x])_X=\varLambda(f_{|[p_x,q_x]},[p_x,q_x])=1\;.
    \end{multline*}
    Since the map $f$ has no fixed points in $\{p_x,q_x\}$, Theorem~\ref{teor punto fijo} implies $\varLambda(f,\{p_x,q_x\})_X=0$, and so $\varLambda(f,(p_x,q_x))_X=1$. Hence, Equation~\eqref{ec del ejemplo} implies $\varLambda(f,X_1)_X=-1$. Let us take now some $x'\in X\setminus [p_x,q_x]$. Let $X_2=X_1\setminus (p_{x'},q_{x'})$. The same argument as before leads to $\varLambda(f,X_2)=-2$. We can repeat this process for every $n\in \mathbb{N}$ (Figure~\ref{generalizado de corona} shows the situation for $X_4$) and, as a consequence, we see that the sequence $|\varLambda(f,X_n)_X|=|-n|=n$ gives us a better approach to the number of fixed points of $f$. \qedhere
\end{example}
\begin{figure}[htb] 
   \centering
    \includegraphics[scale=0.55]{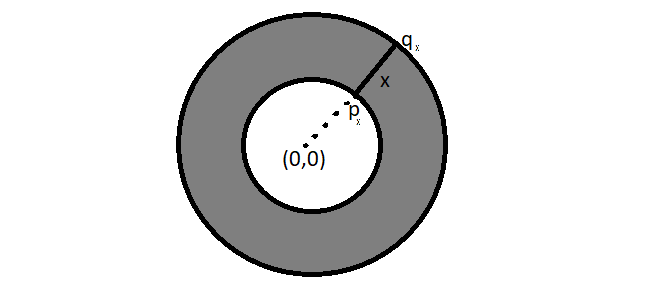} 
    \caption{Complex $X$.}
   \label{corona circular tesis}
\end{figure}
\begin{figure}[htb] 
   \centering
    \includegraphics[scale=0.55]{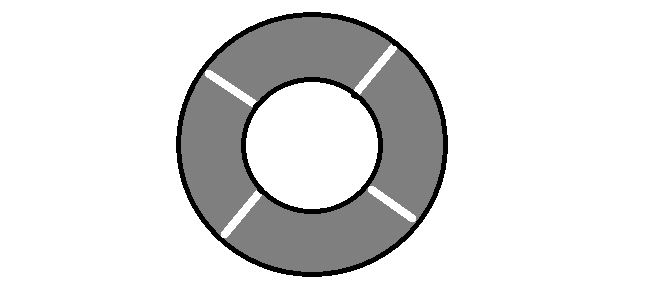} 
    \caption{$X_4$.}
   \label{generalizado de corona}
\end{figure}

As a first application of the topological invariance of the combinatorial Lefschetz number, we obtain a generalization of the classic topological-invariance result of the index by O'Neill \cite[Theorem 2.5 and Corollary 7.2]{ONeill}. In \cite[Theorem 2.9]{A-M-M}, the authors and Mosquera-Lois showed a relationship between the combinatorial Lefschetz number and the fixed-point index. Hence, if the combinatorial Lefschetz number is a topological invariant, we might expect that this invariance will also lead to a topological-invariance result of the index. This is the idea of Theorem~\ref{generalizacion oneil}.

The result in \cite[Theorem 2.5 and Corollary 7.2]{ONeill} can be stated as:
\begin{theorem}\label{teor oneil}
    Let $X$ be a simplicial complex, $f$ a self-map of $X$ and $U\subset X$ an open subspace such that $f$ has no fixed points in the boundary of $U$. If $U\cup f(U)$ is contained in some subcomplex of $X$ homeomorphic under a homeomorphism $h$ to a subcomplex of a simplicial complex $Y$ and $g:Y\rightarrow Y$ satisfies $(g\circ h)_{|U}=(h\circ f)_{|U}$, then $i(X,f,U)=i(Y,g,h(U))$.
\end{theorem}
We have the following theorem that generalizes \cite[Theorem 2.9]{A-M-M}.
\begin{theorem}\label{teor para probar oneil}
    Let $X$ be a simplicial complex, $U\subset X$ a definable open subspace and $f:X\rightarrow X$ a open map such that $f(U)\subset U$, $f(X\setminus U)\subset X\setminus U$ and $f$ has no fixed points in $\overline{U}\setminus U$. Then $\varLambda(f,U)_X=i(X,f,U)$.
\end{theorem}
\begin{proof}
    Due to Theorem~\ref{thm:triangulation}, we can assume that $\overline{U}$ is a subcomplex of $X$. Since $f(U)\subset U$, it makes sense to consider $f_{|\overline{U}}:\overline{U}\rightarrow\overline{U}$. We get
    \begin{equation*}
        \varLambda(f,\overline{U})_X=\varLambda(f,\overline{U})_{\overline{U}}=\varLambda(f,\overline{U})=i(\overline{U},f,U)\;,
    \end{equation*}
    where the first equality follows from \cite[Remark 3.5]{M-M1} and the last one from the normalization and additivity axioms of the index \cite[Chap.~IV]{Brown}. But, using \cite[Corollary 7.2]{ONeill}, we have that $i(\overline{U},f,U)=i(X,f,U)$.
    
    Moreover, the additivity of the combinatorial Lefschetz number gives $\varLambda(f,\overline{U})_X=\varLambda(f,U)_X+\varLambda(f,\overline{U}\setminus U)_X$. But, using Theorem~\ref{teor punto fijo}, we obtain that $\varLambda(f,\overline{U}\setminus U)_X=0$ since, by hypothesis, $f$ has no fixed points in $\overline{\overline{U}\setminus U}=\overline{U}\setminus U$. So the result follows.
\end{proof}
Now we can generalize \cite[Corollary 7.2]{ONeill} when $f(U)\subset U$.
\begin{theorem}\label{generalizacion oneil}
    Let $X$ and $Y$ be simplicial complexes, $f:X\rightarrow X$ and $g:Y\rightarrow Y$ open maps, and $U\subset X$ and $V\subset Y$ open definable subspaces such that $f(U)\subset U$, $g(V)\subset V$, $f(X\setminus U)\subset X\setminus U$, $g(Y\setminus V)\subset Y\setminus V$, $f$ has no fixed points in $\overline{U}\setminus U$ and $g$ has no fixed points in $\overline{V}\setminus V$. If there is a homeomorphism $h:U\rightarrow V$ such that $g_{|V}\circ h=h\circ f_{|U}$, then $i(X,f,U)=i(Y,g,V)$.
\end{theorem}
\begin{proof}
    From Theorem~\ref{teor para probar oneil} we have that $i(X,f,U)=\varLambda(f,U)_X$ and $i(Y,g,V)=\varLambda(g,V)_Y$. But, by Theorem~\ref{teor inv topo}, we conclude that both terms are equal. 
\end{proof}
\begin{remark}
    What generalizes \cite[Corollary 7.2]{ONeill} in Theorem~\ref{generalizacion oneil} is that $U$ and $V$ need not to be contained in compact subcomplexes that are homeomorphic. It is enough that they are homeomorphic.
\end{remark}

\section{First Applications}\label{axiomas y ejemplos}

The Lefschetz number is one of the main tools one can use to detect the existence of fixed points. For this reason, it is crucial to find the most convenient way to compute it. 
In this section, we present some examples where the combinatorial Lefschetz number is a significantly better way of calculating the Lefschetz number than with homological computations or the technique described in \cite{Arkowitz} (which we summarize some lines below).

We will also show that the topological invariance of the combinatorial Lefschetz number can replace two of the axioms in \cite{Arkowitz}. More precisely, we will see that, when the map is open, the topological invariance of the combinatorial Lefschetz number implies the cofibration axiom (Theorem~\ref{inv implica cof}). Indeed, if the map is a homeomorphism, this invariance will imply the wedge-of-cicles axiom too (Idea~\ref{computo en esferas con combinatorio}).
 Besides the relation between the topological invariance of the combinatorial Lefschetz number as a new axiom and the computation of the Lefschetz number, a new axiomatization of the Lefschetz number is important by itself in fixed-point theory. From books \cite[Chap.~IV]{Brown}, \cite[VII.5]{Dold} to research papers \cite{Arkowitz,G-S,Staecker1,Staecker2}, the axiomatization of fixed-point tools has been quite important in fixed-point theory.

We begin the section by discussing the relation between the axioms in \cite{Arkowitz} and the topological invariance of the combinatorial Lefschetz number. Then, we explain how the topological invariance of the combinatorial Lefschetz number can be used to obtain lower bounds for the Nielsen number in connected sums of manifolds and, finally, we illustrate through some examples how the computation of the Lefschetz number becomes much easier using the combinatorial Lefschetz number.
\subsection{A new set of axioms}\label{subseccion axiomas}
The Lefschetz number was characterized in \cite{Arkowitz} as the only function from $\mathcal{C}$ to the integers ($\mathcal{C}$ is the family of pairs $(X,f)$ with $X$ a finite connected CW-complex and $f:X\rightarrow X$ a map) satisfying four axioms: commutativity, homotopy invariance, cofibration rule and initial value on the wedge product of circles. The main idea of the proof is that, from these axioms, we can characterize the Lefschetz number of maps on the wedge product of $n$-dimensional spheres and not only of $1$-spheres. Then, if we have a map $f:X\rightarrow X$ with $X$ a CW-complex of dimension $n$, we consider a cellular approximation $\tilde{f}$ of $f$ (homotopy axiom) and we use the cofibration axiom taking the quotient $\frac{X}{X^{(n-1)}}$, where $X^{(n-1)}$ is the $(n-1)$-skeleton of the complex. Then, the result follows by induction on the dimension, since $\frac{X}{X^{(n-1)}}$ is a wedge product of spheres.

When the map is open (in particular, if the map is a homeomorphism), the topological invariance of the combinatorial Lefschetz number will imply the cofibration axiom in the sense of the next theorem.
\begin{theorem}\label{inv implica cof}
    Let $X$ be a simplicial complex, $f:X\rightarrow X$ an open map and $A\subset X$ a subcomplex {\rm(}in particular definable{\rm)} such that $f(A)=A$ {\rm(}here $f(A)\subset A$ is not enough{\rm)} and $f(X\setminus A)\subset X\setminus A$. Then,
    \begin{equation*}
        \varLambda(f,X)=\varLambda(f_{|A},A)+\varLambda(\tilde{f},X/A)-1\;.
    \end{equation*}
\end{theorem}
\begin{proof}
    By the additivity of the Lefschetz number we have
    \begin{multline}\label{ec para citar1}
        \varLambda(f,X)=\varLambda(f,X)_X=\varLambda(f,A)_X+\varLambda(f,X\setminus A)_X=\\
        \varLambda(f_{|A},A)+\varLambda(f,X\setminus A)_X\;,
    \end{multline}
    where the last equality is due to \cite[Remark 3.5]{M-M1} ($A$ is closed).

    Now, since $f$ is open, $f(A)=A$ and $f(X\setminus A)\subset X\setminus A$, one can check that $\tilde{f}:X/A\rightarrow X/A$ is also open, $\tilde{f}(\{A\})=\{A\}$ ($A\in X/A$) and $\tilde{f}\bigl((X/A)\setminus\{A\}\bigr)\subset (X/A)\setminus\{A\}$.

    Again, by additivity, we have
    \begin{align*}
        &\varLambda(\tilde{f},X/A)=\varLambda(\tilde{f},X/A)_{X/A}=\\
        &\varLambda(\tilde{f},\{A\})_{X/A}+\varLambda(\tilde{f}, (X/A)\setminus\{A\})_{X/A}=
        1+\varLambda(\tilde{f},(X/A)\setminus \{A\})_{X/A}\;.
        \end{align*}
        But the topological invariance of the combinatorial Lefschetz number gives
        \begin{equation*}
            \varLambda(\tilde{f},(X/A)\setminus \{A\})_{X/A}=\varLambda(f,X\setminus A)_X\;,
        \end{equation*}
        so Equation~\eqref{ec para citar1} becomes
        \begin{equation*}
            \varLambda(f,X)=\varLambda (f_{|A}, A)+\varLambda(\tilde{f},X/A)-1\;.\qedhere
        \end{equation*}
\end{proof}

When $f$ is a homemomorphism, not just an open map, the wedge-of-circles axiom in \cite{Arkowitz} is also a consequence of the topological invariance of the combinatorial Lefschetz number. We will illustrate now the main idea of this implication.
\begin{idea}\label{computo en esferas con combinatorio}
\textbf{First case: One circle.} Given a homeomorphism $f:S^1\rightarrow S^1$, \cite[Lemma 2.7.4]{Dieck} gives a pointed map homotopic to $f$. In fact, if $f$ is a homeomorphism, this new map is also a homeomorphism. So, up to homotopy, we can assume that the homeomorphisms of the circle have a fixed point $p$. In fact, we can assume that a homeomorphism of $S^{1}$ is homotopic to the identity $\mathrm{id}$ or to $\mathrm{id}^{-1}$; that is, $\mathrm{deg}(f)=1$ or $\mathrm{deg}(f)=-1$. Here, the wedge-of-circles axiom tells us that the Lefschetz number must be $0$ or $2$. However, we can obtain the same result without this axiom. 

Suppose, for example, that $f=\mathrm{id}$ (the other case is analogous) and let $p$ be a point in $S^1$. The additivity of the combinatorial Lefschetz number gives
\begin{equation*}
    \varLambda(\mathrm{id},S^1)=\varLambda(\mathrm{id},S^{1}\setminus\{p\})_{S^1}+\varLambda(\mathrm{id},p)=\varLambda(\mathrm{id},S^{1}\setminus\{p\})_{S^1}+1\;.
\end{equation*}
Now, since $S^{1}\setminus\{p\}$ is homeomorphic to the interval $(0,1)$ (and the identity in the open interval extends to the identity in the closed interval), due to the topological invariance of the combinatorial Lefschetz number, we have $\varLambda(\mathrm{id},S^{1}\setminus\{p\})_{S^1}=\varLambda(\mathrm{id},(0,1))_{[0,1]}$. But
\begin{multline*}
    1=\varLambda(\mathrm{id},[0,1])=\varLambda(\mathrm{id},(0,1))_{[0,1]}+\varLambda(\mathrm{id},\{0,1\})_{[0,1]}=\\\varLambda(\mathrm{id},(0,1))_{[0,1]}+2
\end{multline*}
(the first equality comes from the simple connectivity of the interval).
 Finally, we get $\varLambda(\mathrm{id},S^1\setminus\{p\})_{S^1}=-1$, and then $\varLambda(\mathrm{id},S^1)=0=1-\mathrm{deg}(f)$.

\textbf{Second case: Two circles.} Imagine now that we have the wedge product $S_1^1\vee S_2^1$ of two circles. Since $f$ is a homeomorphism, it must send the common point $p$ of the two circles to itself. By the additivity of the combinatorial Lefschetz number, we obtain
\begin{multline*}
    \varLambda(f,S_1^1\vee S_2^1)=\varLambda(f,S_1^1\vee S_2^1\setminus \{p\})_{S_1^1\vee S_2^1}+\varLambda(f,\{p\})_{S_1^1\vee S_2^1}\\=\varLambda(f,S_1^1\vee S_2^1\setminus \{p\})_{S_1^1\vee S_2^1}+1\;.
\end{multline*}
Now, $S_1^1\vee S_2^1\setminus \{p\}$ is homeomorphic to the disjoint union of two open intervals $\mathring{\mathrm{I}}_1\sqcup\mathring{\mathrm{I}}_2$, where $f$ induces a map $g:\mathring{\mathrm{I}}_1\sqcup\mathring{\mathrm{I}}_2\rightarrow\mathring{\mathrm{I}}_1\sqcup\mathring{\mathrm{I}}_2$.
Then it can happen that $g$ maps $\mathring{\mathrm{I}}_0$ to $\mathring{\mathrm{I}}_1$ and $\mathring{\mathrm{I}}_1$ to $\mathring{\mathrm{I}}_0$, or $\mathring{\mathrm{I}}_0$ to $\mathring{\mathrm{I}}_0$ and $\mathring{\mathrm{I}}_1$ to $\mathring{\mathrm{I}}_1$. In the first case, $f$ must map $S^1_1\setminus\{p\}$ to $S^1_2\setminus\{p\}$ and $S^1_2\setminus\{p\}$ to $S^1_1\setminus\{p\}$ and, in the second case, it must map $S^1_1\setminus\{p\}$ to $S^1_1\setminus\{p\}$ and $S^1_2\setminus\{p\}$ to $S^1_2\setminus\{p\}$.

In the first case (note that $g$ can be extended as a homeomorphism to $\mathrm{I}_1\sqcup\mathrm{I}_2$), independently of the orientation of the $g_i$'s, we have 
\begin{equation*}
    \varLambda(g,\mathring{\mathrm{I}}_1\sqcup\mathring{\mathrm{I}}_2)_{\mathrm{I}_1\sqcup\mathrm{I}_2}=0
\end{equation*}
since the space is compact and there are no fixed points. Hence, due to the topological invariance of the combinatorial Lefschetz number, we have $\varLambda(f,S_1^1\vee S_2^1)=0+1=1-\mathrm{deg}(f_1)-\mathrm{deg}(f_2)$, where $f_i$ denotes the restriction of $f$ to $S^1_i$ composed with the projection to $S^1_i$.
In the second case, by \cite[Remark 3.5]{M-M1}, we have
\begin{equation*}
    \varLambda(f,S^1_1)_{S^1_1\vee S^1_2}=\varLambda(f_{|S^1_1},S^1_1)\;,
\end{equation*}
and we have already seen that $\varLambda(f_{|S^1_1},S^1_1)=1-\mathrm{deg}(f_1)=1-\mathrm{deg}(f_{|S^1_1})$. Similarly, we have $\varLambda(f_{|S^1_2},S^1_2)=1-\mathrm{deg}(f_2)=1-\mathrm{deg}(f_{|S^1_2})$.

Finally, since
\begin{multline*}
    \varLambda(f,S^1_1\vee S^1_2)=\\
    \varLambda(f,S^1_1\setminus\{p\})_{S^1_1\vee S^1_2}+\varLambda(f,\{p\})_{S^1_1\vee S^1_2}+\varLambda(f,S^1_2\setminus\{p\})_{S^1_1\vee S^1_2}\;,
\end{multline*}
\begin{equation*}
    \varLambda(f_{|S^1_1},S^1_1)=\varLambda(f_{|S^1_1},S^1_1\setminus\{p\})_{S^1_1}+\varLambda(f_{|S^1_1},\{p\})_{S^1_1}\;,
\end{equation*}
\begin{equation*}
    \varLambda(f_{|S^1_2},S^1_2)=\varLambda(f_{|S^1_2},S^1_2\setminus\{p\})_{S^1_2}+\varLambda(f_{|S^1_2},\{p\})_{S^1_2}\;,
\end{equation*}
and, by \cite[Remark 3.5]{M-M1},
\begin{equation*}
\varLambda(f,S^1_1\setminus\{p\})_{S^1_1\vee S^1_2}=\varLambda(f_{|S^1_1},S^1_1\setminus\{p\})_{S^1_1}\;,
\end{equation*}
\begin{equation*}
\varLambda(f,S^1_2\setminus\{p\})_{S^1_1\vee S^1_2}=\varLambda(f_{|S^1_2},S^1_2\setminus\{p\})_{S^1_2}\;,
\end{equation*}
\begin{equation*}
    \varLambda(f,\{p\})_{S^1_1\vee S^1_2}=\varLambda(f_{|S^1_1}, \{p\})_{S^1_1}=\varLambda(f_{|S^1_2}, \{p\})_{S^1_2}=1\;,
\end{equation*}
we conclude that 
\begin{equation*}
\varLambda(f,S^1_1\vee S^1_2)=1-\mathrm{deg}(f_1)-\mathrm{deg}(f_2)\;.
\end{equation*}

\textbf{In general}, if we have a wedge product of $n$ circles, we divide the study in two cases. If $f$ maps a circle to itself, we have that the Lefschetz number of the total space equals the Lefschetz number at this circle (caracterized in our first case) plus the Lefschetz number of the wedge product of the other $n-1$ circles (here we use a recurrence argument) minus one (since the common point is considered twice).

If no circle is sent to itself, as in the case of $S^1_1\vee S^1_2$, it is easy to see that the Lefschetz number is $1=1-\mathrm{deg}(f_1)+\ldots+\mathrm{deg}(f_2)$.
\end{idea}

\subsection{A bound for some Nielsen numbers on the connected sum of tori}\label{seccion cotas nielsen}
Given an suitable simplicial complex $X$ (see conditions in \cite[Theorem 5.3]{Schirmer}) and two subcomplexes $A_1$ and $A_2$ such that $X=A_1\cup A_2$, the Nielsen number of a triad (which we will write $N(f;A_1\cup A_2)$, where $f:X\rightarrow X$ is a map that preserves $A_1$ and $A_2$), defined in \cite{Schirmer}, can be seen as the minimum number of fixed points of any map homotopic to $f$ under a homotopy that preserves $A_1$ and $A_2$. It is a generalization of the relative Nielsen number and it equals the Nielsen number in the conditions of \cite[Theorem 4.16]{Heath}.

In this subsection, we will present a very illustrative and relevant application of the combinatorial Lefschetz number. In Corollary~\ref{corolario nielsen}, we give in the connected sum of two tori a lower bound for the \textit{Nielsen number of a triad}. A similar result may work on the connected sum of more tori if we can define a generalization of the Nielsen number of a triad in such a way that \cite[Theorem 4.12]{Schirmer} has an appropriate generalization.


Indeed, we wonder if similar ideas could be used to compute other different Nielsen numbers since, as we will see, the reason to choose the Nielsen number of a triad is that it satisfies a kind of additivity property. Moreover, if the Nielsen number of any other manifold has an expression in terms of the Lefschetz number, this process may be applied for the connected sum of these manifolds. 


\begin{notation}
    In this section, $\mathbb{T}^p$ will denote the $p$-torus $S^1\times\,\overset{p}{\cdots}\,\times S^1$ and $n\mathbb{T}^p$ the connected sum of $n$ $p$-tori. Given $2\mathbb{T}^p$, we can describe it as the disjoint union of a sphere $S^{p-1}$ and two $p$-tori with a $p$-disk removed. We will denote by $\mathbb{T}^p_1$ and $\mathbb{T}^p_2$ each of these two open tori. Moreover, we will denote by $\overline{\mathbb{T}^p_1}$ and $\overline{\mathbb{T}^p_2}$ the closures of $\mathbb{T}^p_1$ and $\mathbb{T}^p_2$ in $2\mathbb{T}^p$.
\end{notation}

\begin{notation}\label{notacion extension por un punto}
Let $\tilde{f}_1:\mathbb{T}^p\rightarrow \mathbb{T}^p$ and $\tilde{f}_2:\mathbb{T}^p\rightarrow \mathbb{T}^p$ be the extensions of $f_{|\mathbb{T}^p_1}$ and $f_{|\mathbb{T}^p_2}$ to their respective Alexandroff's compactifications (here $\mathbb{T}^p$ is the $p$-torus, the one-point compactification of $\mathbb{T}^p_i$, for $i=1,2$). 
\end{notation}
\begin{notation}
    Given a simplicial complex $X$, simplicial subcomplexes $A_1$, $A_2$ of $X$ such that $X=A_1\cup A_2$, and a map $f:X\rightarrow X$ such that $f(A_i)\subset A_i$ (for $i=1,2$), we will denote by $N(f;X)$ the classical Nielsen number of $f$,  by $N(f;X,A_1)$ the relative Nielsen number (see \cite{Zhao} or \cite[Sec.~3]{Schirmer}) and by $N(f;A_1\cup A_2)$ the Nielsen number of a triad defined in \cite{Schirmer}.
\end{notation}


\begin{theorem}\label{idea acotacion nielsen}
    Let $p\geq 3$ and $f:2\mathbb{T}^p\rightarrow 2\mathbb{T}^p$ be a map such that $f(\mathbb{T}^p_1)\subset\mathbb{T}^p_1$, $f(\mathbb{T}^p_2)\subset\mathbb{T}^p_2$ and $f(S^{p-1})\subset S^{p-1}$. Then,
    \begin{equation*}
       N(f;\overline{\mathbb{T}^p_1}\cup \overline{\mathbb{T}^p_2})\geq  |\varLambda(\tilde{f}_1,\mathbb{T}^p)| + |\varLambda(\tilde{f}_2,\mathbb{T}^p)|-3\;.
    \end{equation*}
\end{theorem}
\begin{proof}
    If we are in the second case of \cite[Theorem 4.12]{Schirmer}, we have
    \begin{equation*}
        N(f;\overline{\mathbb{T}^p_1}\cup \overline{\mathbb{T}^p_2})= N(f_{|\overline{\mathbb{T}^p_1}};\overline{\mathbb{T}^p_1},S^{p-1})+N(f_{|\overline{\mathbb{T}^p_2}};\overline{\mathbb{T}^p_2},S^{p-1})-1\;.
    \end{equation*}
    Indeed, if we are in the first situation of \cite[Theorem 4.12]{Schirmer} (in particular $N(f_{|S^{p-1}},S^{p-1})=0$ and hence $f_{|S^{p-1}}$ has no essential fixed-point classes), we have $N(f_{|\overline{\mathbb{T}^p_1}},f_{|S^{p-1}})=0$ and $N(f_{|\overline{\mathbb{T}^p_2}},f_{|S^{p-1}})=0$ (see \cite[Definition 2.3]{Zhao} for the definition of these two terms) and, from the definition of relative Nielsen number \cite[Definition 2.5]{Zhao}, this implies
    \begin{equation*}
        N(f_{|\overline{\mathbb{T}^p_1}};\overline{\mathbb{T}^p_1})=N(f_{|\overline{\mathbb{T}^p_1}};\overline{\mathbb{T}^p_1},S^{p-1}),\quad N(f_{|\overline{\mathbb{T}^p_2}};\overline{\mathbb{T}^p_2})=N(f_{|\overline{\mathbb{T}^p_2}};\overline{\mathbb{T}^p_2},S^{p-1})\;.
    \end{equation*}
So, in this case, we obtain 
    \begin{equation*}
        N(f;\overline{\mathbb{T}^p_1}\cup \overline{\mathbb{T}^p_2})= N(f_{|\overline{\mathbb{T}^p_1}};\overline{\mathbb{T}^p_1},S^{p-1})+N(f_{|\overline{\mathbb{T}^p_2}};\overline{\mathbb{T}^p_2},S^{p-1})\;.
    \end{equation*}
    Consequently, we will always have
 \begin{equation}\label{aditividad num nielsen}
        N(f;\overline{\mathbb{T}^p_1}\cup \overline{\mathbb{T}^p_2})\geq N(f_{|\overline{\mathbb{T}^p_1}};\overline{\mathbb{T}^p_1},S^{p-1})+N(f_{|\overline{\mathbb{T}^p_2}};\overline{\mathbb{T}^p_2},S^{p-1})-1\;.
 \end{equation}
    Since the three conditions in \cite[Theorem 6.7]{Zhao} are satisfied by $(\overline{\mathbb{T}^p_1},S^{p-1})$ and $(\overline{\mathbb{T}^p_2},S^{p-1})$ (recall that $p\geq 3$), we have that $N(f_{|\overline{\mathbb{T}^p_1}};\overline{\mathbb{T}^p_1},S^{p-1})$ (resp., $N(f_{|\overline{\mathbb{T}^p_2}};\overline{\mathbb{T}^p_2},S^{p-1})$) can be given as the minimum number of fixed points of the maps homotopic to $f_{|\overline{\mathbb{T}^p_1}}$ (resp., to $f_{|\overline{\mathbb{T}^p_2}}$) relative to $S^{p-1}$. But, in this case, since $\mathbb{T}^p=\overline{\mathbb{T}^p_1}/S^{p-1}=\overline{\mathbb{T}^p_2}/S^{p-1}$ also has the Wecken property (and so, $N(\tilde{f_i};\mathbb{T}^p)$ can be given as the minimum number of fixed points of the maps homotopic to $\tilde{f}_i$), we see that 
\begin{align*}
    &N(\tilde{f_1};\mathbb{T}^p)\leq N(f_{|\overline{\mathbb{T}^p_1}};\overline{\mathbb{T}^p_1},S^{p-1})+1,\\
    &N(\tilde{f_2};\mathbb{T}^p)\leq N(f_{|\overline{\mathbb{T}^p_2}};\overline{\mathbb{T}^p_2},S^{p-1})+1\;.
\end{align*}
(Recall the definition of $\tilde{f}_i$ in Notation~\ref{notacion extension por un punto}.)

So, from Equation~\eqref{aditividad num nielsen}, we obtain
\begin{equation*}
     N(f;\overline{\mathbb{T}^p_1}\cup \overline{\mathbb{T}^p_2})\geq N(\tilde{f_1};\mathbb{T}^p)+N(\tilde{f_2};\mathbb{T}^p)-3\;.
\end{equation*}
From \cite{B-B-P-T}, this last term equals $| \varLambda(\tilde{f}_1,\mathbb{T}^p)|+|\varLambda(\tilde{f}_2,\mathbb{T}^p)|-3$.
\end{proof}
Now the topological invariance of the combinatorial Lefschetz number allows us to improve this result.
\begin{corollary}\label{corolario nielsen}
    In the conditions of the previous theorem, if indeed $f$ $\tilde{f}_1$ and $\tilde{f}_2$ are open and $N(\tilde{f}_i)=\varLambda(\tilde{f}_i,\mathbb{T}^p)$ for both $i=1,2$, then,
    \begin{equation*}
        N(f;\overline{\mathbb{T}^p_1}\cup \overline{\mathbb{T}^p_2})\geq\varLambda(f,2\mathbb{T}^p)-\varLambda(f_{|S^{p-1}},S^{p-1})-1\;.
    \end{equation*}
    If $N(\tilde{f}_i)=-\varLambda(\tilde{f}_i,\mathbb{T}^p)$ for both $i=1,2$, then
    \begin{equation*}
        N(f;\overline{\mathbb{T}^p_1}\cup \overline{\mathbb{T}^p_2})\geq-\varLambda(f,2\mathbb{T}^p)+\varLambda(f_{|S^{p-1}},S^{p-1})-5\;.
    \end{equation*}
 \end{corollary}
\begin{proof}
    We prove the first case. From Theorem~\ref{idea acotacion nielsen}, we already have
\begin{equation}\label{segunda ec nielsen}
     N(f;\overline{\mathbb{T}^p_1}\cup \overline{\mathbb{T}^p_2})\geq \varLambda(\tilde{f}_1,\mathbb{T}^p)+\varLambda(\tilde{f}_2,\mathbb{T}^p)-3\;.
\end{equation}
Let us write $\infty$ for the point at infinite in the Alexandroff's compactification of $\mathbb{T}^p_i$, that is, the torus $\mathbb{T}^p$. By the additivity of the combinatorial Lefschetz number, we have
\begin{equation*}
    \varLambda(\tilde{f}_i,\mathbb{T}^p)=\varLambda(\tilde{f}_i,\mathbb{T}^p\setminus\{\infty\})_{\mathbb{T}^p}+1
\end{equation*}
for $i=1,2$ (note that we can consider the combinatorial Lefschetz number since, by hypothesis, $\tilde{f}_i$ is open for $i=1,2$). But the topological invariance of the combinatorial Lefschetz number implies that $\varLambda(\tilde{f}_i,\mathbb{T}^p\setminus\{\infty\})_{\mathbb{T}^p}=\varLambda(f,\mathbb{T}^p_i)_{2\mathbb{T}^p}$, and, again from the additivity of the combinatorial Lefschetz number, we have
\begin{equation*}
\varLambda(f,2\mathbb{T}^p)=\varLambda(f,2\mathbb{T}^p)_{2\mathbb{T}^p}=\varLambda(f,\mathbb{T}^p_1)_{2\mathbb{T}^p}+\varLambda(f,\mathbb{T}^p_2)_{2\mathbb{T}^p}+\varLambda(f,S^{p-1})_{2\mathbb{T}^p}\;,
\end{equation*}
and so 
\begin{equation*}
    \varLambda(f,\mathbb{T}^p_1)_{2\mathbb{T}^p}+\varLambda(f,\mathbb{T}^p_2)_{2\mathbb{T}^p}=\varLambda(f,2\mathbb{T}^p)-\varLambda(f_{|S^{p-1}},S^{p-1})\;.
\end{equation*}
Thus, Equation~\eqref{segunda ec nielsen} becomes
\begin{equation*}
    N(f;\overline{\mathbb{T}^p_1}\cup \overline{\mathbb{T}^p_2})\geq\varLambda(f,2\mathbb{T}^p)-\varLambda(f_{|S^{p-1}},S^{p-1})-1\;.
\end{equation*}

The second case is similar.
\end{proof}

\begin{remark}
    In the case of the connected sum of several tori, using the Nielsen number of a triad we only have
    \begin{align*}
        &N\bigl(f;\overline{\mathbb{T}^p_1}\cup(\overline{\mathbb{T}^p_2}\cup\ldots\cup\overline{\mathbb{T}^p_k})\bigr)\geq\\
        &N(f_{|\overline{\mathbb{T}^p_1}};\overline{\mathbb{T}^p_1},S^{p-1})        +N(f_{|\overline{\mathbb{T}^p_2}\cup\ldots\cup\overline{\mathbb{T}^p_k}};\overline{\mathbb{T}^p_2}\cup\ldots\cup\overline{\mathbb{T}^p_k},S^{p-1})-1
        \geq\\
        &N(\tilde{f}_1;\mathbb{T}^p)+N(\tilde{f}_2;(k-1)\mathbb{T}^p)
        -3=
        |\varLambda(\tilde{f}_1)|+N(\tilde{f}_2;(k-1)\mathbb{T}^p)-3.
    \end{align*}
    (Here, $\tilde{f}_2$ is the extension of $f_{|\overline{\mathbb{T}^p_2}\cup\ldots\cup\overline{\mathbb{T}^p_k}}$ to its one-point compactification.)

    However, if the Nielsen number of a triad could be generalized to the union of more than two subspaces having a kind of additive property as in \cite[Theorem 4.12]{Schirmer}, we would obtain similar bounds for $N(f;\overline{\mathbb{T}^p_1}\cup \ldots \cup \overline{\mathbb{T}^p_k})$ in the connected sum of $k$ $p$-tori. 

    In \cite{Heath}, a generalization of the Nielsen number of a triad is given for a family of subspaces indexed by an index set $M$. However, in our case, the analogue in that article to \cite[Theorem 4.12]{Schirmer} (that is, \cite[Corollary 4.5]{Heath}) does not work since, among other problems, the condition of being essentially reducible is too strong. Nevertheless, we believe that, limiting the generalization of the Nielsen number of a triad to a finite ``ad'', it would be possible to obtain a good generalization of \cite[Theorem 4.12]{Schirmer} for our purpose.
\end{remark}
\subsection{Some Examples of computation}\label{subseccion de los ejemplos}

Now we present some examples that illustrate the usefulness of the combinatorial Lefschetz number when computing the Lefschetz number. For simplicity, we are computing the Lefschetz number of homeomorphisms, but in some cases it is also possible to use the same arguments for open maps. More examples will be given in Section~\ref{seccion aplicaciones practicas}.

Among the examples we present, many of them are theoretical (with no explicit homeomorphism ---although if we describe it---). However, we also present an explicit example (Example~\ref{ej cono raro explicito}). Whereas in the first examples, the only topological-invariance result that we will need is that of \cite[Remark 3.5]{M-M1}, in Examples~\ref{ej klein toro} and \ref{ej fibrado klein}, the topological invariance of the combinatorial Lefschetz number, in the sense of Theorem~\ref{teor inv topo}, becomes the main step in the simplification of the computations. In all of these cases, the computation of the Lefschetz number is much easier using the combinatorial Lefschetz number.

Note also that some of the hypotheses that we consider in the examples (for instance, to preserve $S$ in Example~\ref{ej klein toro}) are not so restrictive since the homotopy axiom of the Lefschetz number sometimes will allow us to obtain an equivalent map satisfying these hypotheses.


\begin{example}\label{ej cono raro}
    Let $Y$ be a connected simplicial complex. 
    Consider the cone $C$ over $Y$. Now we glue a cylinder $D$ to the base of $C$ by using an embedding from the union of the base and the top of $D$ to the base of $C$. Figure~\ref{figura cono raro} shows the resulting space, denoted by $X$. Finally, let $f:X\rightarrow X$ be a homeomorphism such that $f(C)=C$ and $f(D)=D$ (and hence $f$ maps the union of the top and the base of the cylinder to itself). 

    Using additivity we have
    \begin{align*}
        &\varLambda(f,X)=\varLambda(f,C)_X+\varLambda(f,X\setminus C)_X=
        \varLambda(f_{|C},C)+\varLambda(f,X\setminus C)_X=\\
        &1+\varLambda(f,X\setminus C)_X\;,
    \end{align*}
    where the last equality follows from the contractibility of $C$. Note that $X\setminus C$ is the open cylinder $D\setminus(T\cup B)$, where $T$ and $B$ denote the top and bottom of $D$. Using the definition of $f$ on $T$ and $B$, we can extend $f_{|X\setminus C}$ to the homeomorphism $f_{|D}:D\rightarrow D$. Consequently, due to the partial result of topological invariance in \cite[Remark 3.5]{M-M1}, we have $\varLambda(f,X\setminus C)_X=\varLambda(f_{|D},X\setminus C)_D$. Again, using additivity, we obtain
    \begin{equation*}
        \varLambda(f_{|D},D)=\varLambda(f_{|D},D)_D=\varLambda(f_{|D},X\setminus C)_D+\varLambda(f_{|D},T\cup B)_D\;,
    \end{equation*}
    so
    \begin{align*}
        &\varLambda(f_{|D},X\setminus C)_D=\varLambda(f_{|D},D)-\varLambda(f_{|D},T\cup B)_D=\\
        &\varLambda(f_{|D},D)-\varLambda(f_{|T\cup B},T\cup B)_{T\cup B}\;.
    \end{align*}
    Therefore
    \begin{equation*}
        \varLambda(f,X)=1+\varLambda(f_{|D},D)-\varLambda(f_{T\cup B},,T\cup B)_{T\cup B}\;.
    \end{equation*}
    But the Lefschetz number of a homeomorphism of a disjoint union of two circles is very easy to compute. Consequently, we have reduced the problem of computing the Lefschetz number of $f$ to computing the Lefschetz number of a homeomorphism of the cylinder, which is much easier.
\end{example}

\begin{figure}[htb] 
    \centering
     \includegraphics[scale=0.35]{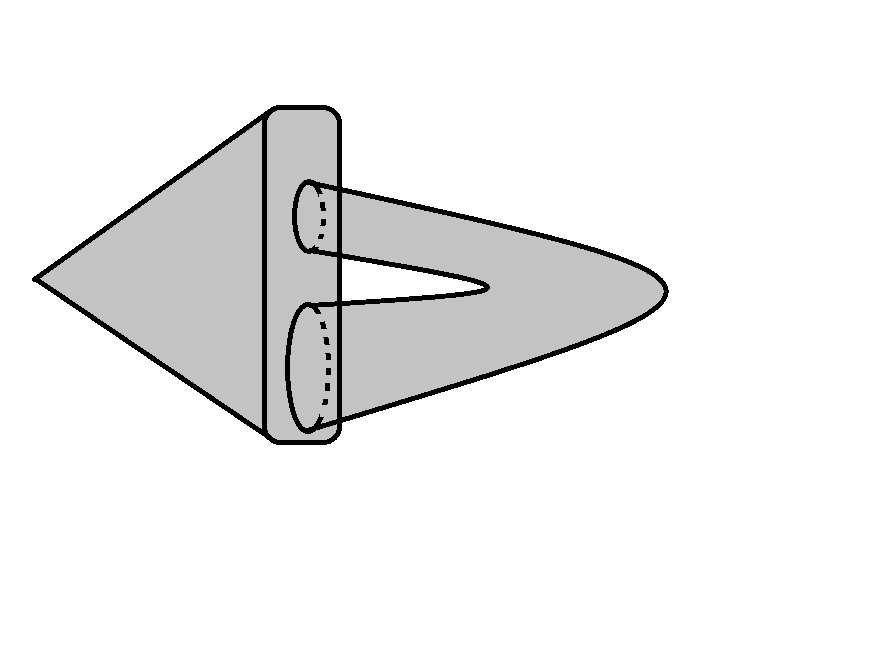} 
     \caption{Space $X$.}
    \label{figura cono raro}
\end{figure}
\begin{remark}
    Note that a similar idea can be applied if, instead of the cone, we glue the cylinder to a space $Z$ where the Lefschetz number of $f_{|Z}$ is known.
\end{remark}
\begin{remark}\label{obs ejemplo cono y cilindro}
    In Example~\ref{ej cono raro}, one could think that if we are willing to compute the Lefschetz number on the cylinder, we can use the cofibration axiom. But, in this case, we obtain
    \begin{equation*}
        \varLambda(f,X)=\varLambda(f_{|D},D)+\varLambda(\tilde{f},\frac{X}{D})-1\;,
    \end{equation*}
    where $\tilde{f}$ is the map induced by $f$ on the quotient. Now, the Lefschetz number of $\tilde{f}$ can be very difficult to compute. The other option (to take the quotient $\frac{X}{C}$) also gives us a space where the Lefschetz number is more difficult to compute than on the cylinder.
\end{remark}
We are going to illustrate Example~\ref{ej cono raro} with an explicit space and map.
\begin{example}\label{ej cono raro explicito}
    Let $Y$ be the planar complex of Figure~\ref{figura base Y}. We can consider $Y$ as a region in the hyperplane $x=0$ in $\mathbb{R}^3$ so that $p$ corresponds to $(0,0,0)$. Consider now the cone $C$ over $Y$ in $\mathbb R^3$, where the vertex of the cone is the point $c=(-1,0,0)$. Finally, we glue the top and the base of the cylinder to $C$ in such way that the resulting space $X$ is symmetric under a $180\;^{\circ}$ rotation over the edge $x$. Let $f$ denote this rotation.

    From Example~\ref{ej cono raro}, we know that 
    \begin{equation*}
        \varLambda(f,X)=1+\varLambda(f_{|D},D)-\varLambda(f_{|T\cup B},T\cup B)_{T\cup B}\;.
    \end{equation*}
    But $\varLambda(f_{|T\cup B},T\cup B)_{T\cup B}=0$ since $f$ sends $T$ to $B$ and $B$ to $T$. So, it only remains to compute $\varLambda(f_{|D},D)$.
    Now note that $f$ sends the ``middle circle'' $S$ of the cylinder (the one that is highlighted in Figure~\ref{figura cono sobre Y}) to itself. Indeed, $f_S$ is a reflection over one diameter of $S$, so $\varLambda(f_{|S},S)=2$.

    Now, consider the retraction $r:D\rightarrow S$ from the cylinder to the middle circle. Since $f_{|D}$ is homotopic to $ f_{|S}\circ r$ we have that $\varLambda(f_{|D},D)=\varLambda(i\circ f_{|S}\circ r,D)$. Indeed, the following diagram commutes:

    \[\begin{tikzcd}
	D & S \\
	D & S\;.
	\arrow["r", from=1-1, to=1-2]
	\arrow["i\circ f_{|S}\circ r"', from=1-1, to=2-1]
	\arrow["{i\circ f_{|S}}"', from=1-2, to=2-1]
	\arrow["f_{|S}"', from=1-2, to=2-2]
	\arrow["r", from=2-1, to=2-2]
\end{tikzcd}\]
    Hence, by \cite[Properties 2.4.1]{Gorniewicz}, we have that 
    \begin{equation*}
    \varLambda(f_{|D},D)=\varLambda(f_{|S},S)=2\;,
    \end{equation*}
and so $\Lambda(f,X)=3$.
\end{example}
\begin{figure}[htb] 
    \centering
     \includegraphics[scale=0.50]{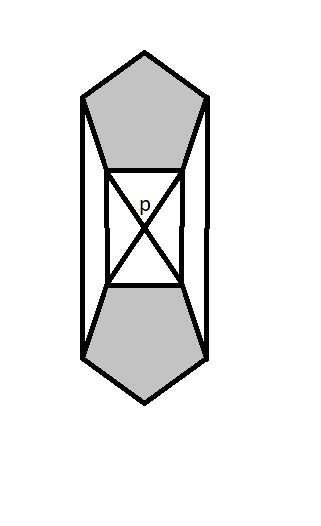} 
     \caption{Space $Y$.}
    \label{figura base Y}
\end{figure}
\begin{figure}[htb] 
    \centering
     \includegraphics[scale=0.35]{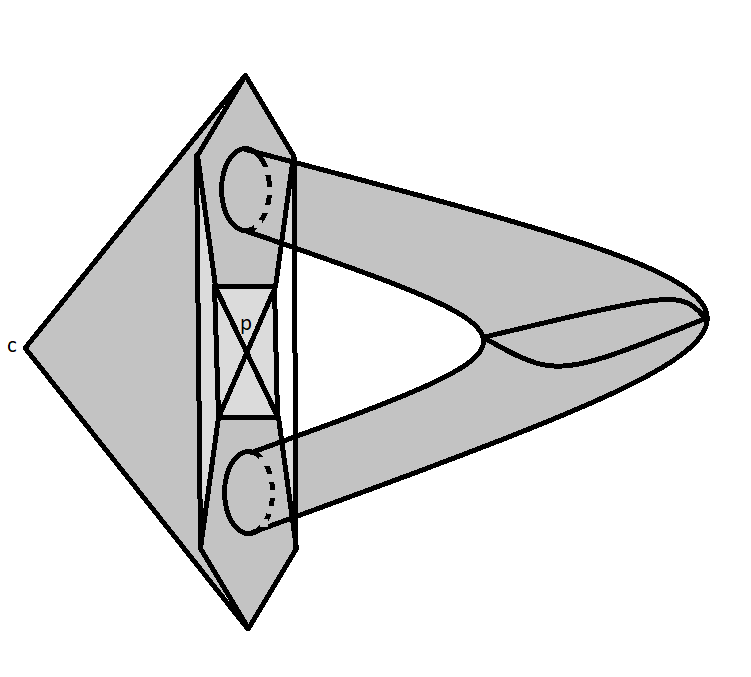} 
     \caption{Space $X$.}
    \label{figura cono sobre Y}
\end{figure}

The idea of the next example is similar to that of Example~\ref{ej cono raro}. However, we found it interesting since, even if the sphere (that is a non trivial space where the axioms of \cite{Arkowitz} work directly) appears and plays an important role, the computation of the Lefschetz number of the whole space using the axioms in \cite{Arkowitz} is not optimal. 
\begin{example}\label{ej cono esfera}
Let $C$ be a cone with base the disk $D^2$ and consider a sphere $S^2$. Let $\alpha\subset S^2$ be the subspace, represented in Figure~\ref{figura cono con esfera}, consisting of the equator and a meridian glued by two points. Let $X$ be the result of gluing $C$ and $S^2$ through an embedding from $\alpha$ to $D^2$. Finally, let $f:X\rightarrow X$ be a homeomorphism such that $f(C)=C$ and $f(S^2)=S^2$ (and hence, $f(\alpha)=\alpha$, where we are identifying $\alpha$ with its image in $X$).

If we use the cofibration axiom, we have
\begin{equation*}
    \varLambda(X,f)=\varLambda(f_{|C},C)+\varLambda(\tilde{f},\frac{X}{C})-1=\varLambda(\tilde{f},\frac{X}{C})\;,
\end{equation*}
where $\tilde{f}$ is the map induced by $f$ in the quotient. Now, $\frac{X}{C}$ is precisely $\frac{S^2}{\alpha}$, which is a wedge product of four $S^2$'s. Although we can deal with this situation without many troubles, we will see now that, using the combinatorial Lefschetz number, the computations are simpler. 

Let $p$ and $q$ be the intersection points of the equator with the meridian that form the loop $\alpha$. Since $f$ is a homeomorphism and $f(\alpha)=\alpha$, we get $f(\{p,q\})=\{p,q\}$ and $f(\alpha\setminus \{p,q\})=\alpha\setminus \{p,q\}$. Now, using additivity, the fact that the Lefschetz number of a contractible space is $1$ and \cite[Remark 3.5]{M-M1}, we have
\begin{align*}
    &\varLambda(f,X)=\varLambda(f,X)_X=\varLambda(f,C)_X+\varLambda(f,S^2\setminus \alpha)_X=\\
    &\varLambda(f,C)+\varLambda(f_{|S^2},S^2\setminus\alpha)_{S^2}=1+\varLambda(f_{|S^2},S^2)_{S^2}-\varLambda(f_{|S^2},\alpha)_{|S^2}=\\
    &1+\varLambda(f_{|S^2},S^2)-\varLambda(f_{|\alpha},\alpha)\;.
\end{align*}
So, instead of computing the Lefschetz number of a homeomorphism in a wedge product of four $S^2$'s, we must compute now the Lefschetz number of a homeomorphism of the sphere (where it is given by the degree due to \cite{Arkowitz}) and $\varLambda(f_{|\alpha},\alpha)$. But, since $f(\{p,q\})=\{p,q\}$, using additivity we have:
\begin{equation*}
    \varLambda(f_{|\alpha},\alpha)=\varLambda(f_{|\alpha},\{p,q\})+\varLambda(f_{|\alpha},\alpha\setminus\{p,q\})\;.
\end{equation*}
The first term of the right-hand side is straightforward once we know $f$. For the second term, which is homeomorphic to a disjoint union of four open intervals, we can repeat the argument of Idea~\ref{computo en esferas con combinatorio} for four circles. So we see that, using the Lefschetz combinatorial number, sometimes we also decrease the dimension of the spaces where we make the computations.

\end{example}
\begin{figure}[htb] 
    \centering
     \includegraphics[scale=0.35]{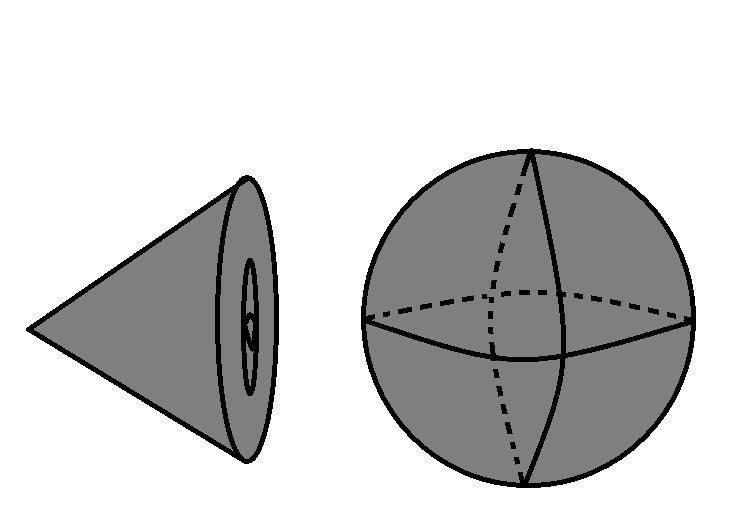} 
     \caption{The gluing of the cone with the sphere.}
    \label{figura cono con esfera}
\end{figure}

The next example shows that sometimes the computation of the Lefschetz number in the Klein bottle reduces to the computation of the Lefschetz number in the cylinder.

\begin{example}\label{ej klein toro}
    Let $K$ be the Klein bottle, considered as quotient of the square described in the first step of Figure~\ref{figura klein toro}. Imagine $f:K\rightarrow K$ is a homeomorphism such that there exists a vertical line $S$ of the square (the dotted line in Figure~\ref{figura klein toro}) such that $f(S)= S$ (note that $S$ is, under the identifications, a circle). 
    
    In this case, using additivity, we have
    \begin{equation*}
        \varLambda(f,K)=\varLambda(f,S)_K+\varLambda(f,K\setminus S)_K=
        \varLambda(f_{|S},S)+\varLambda(f,K\setminus S)_K\;.
    \end{equation*}
    But we know that the Lefschetz number of any homeomorphism of the circle must be $0$ or $2$ (it is deduced from \cite{Arkowitz} since the degree of any homeomorphism can only be $1$ or $-1$). Depending on the orientation of the map, we should choose the correct one.

    In the other hand, $K\setminus S$ is homeomorphic to an open cylinder without the top and the base (see Figure~\ref{ej klein toro}). A compactification of this space can be the compact cylinder, where it is possible to extend $f_{|K\setminus S}$ to a homeomorphism $\tilde{f}$, defining $\tilde{f}$ on the top and on the base of the cylinder as $f_{|S}$ (and sending the top to the top and the base to the base, or the top to the base and the base to the top, depending on the appearance of $f$). Using the topological invariance of the combinatorial Lefschetz number, if $C$ is such a cylinder, we would have:
    \begin{equation*}
        \varLambda(f,K\setminus S)_K=\varLambda(\tilde{f},C')_C\;,
    \end{equation*}
    where $C'$ is the cylinder without the top and the bottom. Indeed, if we write $T$ and $B$ for the top and the bottom, respectively, we have
    \begin{align*}
        &\varLambda(\tilde{f},C)=\varLambda(\tilde{f},C)_C=\varLambda(\tilde{f},C')_C+\varLambda(\tilde{f},\{T,B\})_C=\\
        &\varLambda(\tilde{f},C')_C+\varLambda(\tilde{f}_{|\{T,B\}},\{T,B\})\;.
    \end{align*} 
    But, by the definition of $\tilde{f}$, $\varLambda(\tilde{f}_{|\{T,B\}},\{T,B\})$ can be only $0$ (if $\tilde{f}$ swaps $T$ and $B$) or $2\varLambda(f_{|S},S)$, and hence we obtain
    \begin{equation*}
        \varLambda(f,K)=\varLambda(f_{|S},S)+\varLambda(\tilde{f},C)
    \end{equation*}
    or
    \begin{equation*}
        \varLambda(f,K)=\varLambda(\tilde{f},C)-\varLambda(f_{|S},S)\;.
    \end{equation*}
    So we have reduced the calculation of the Lefschetz number in the Klein bottle to the calculation in the cylinder (and in the circle $S$).
\end{example}

\begin{figure}[htb] 
    \centering
     \includegraphics[scale=0.35]{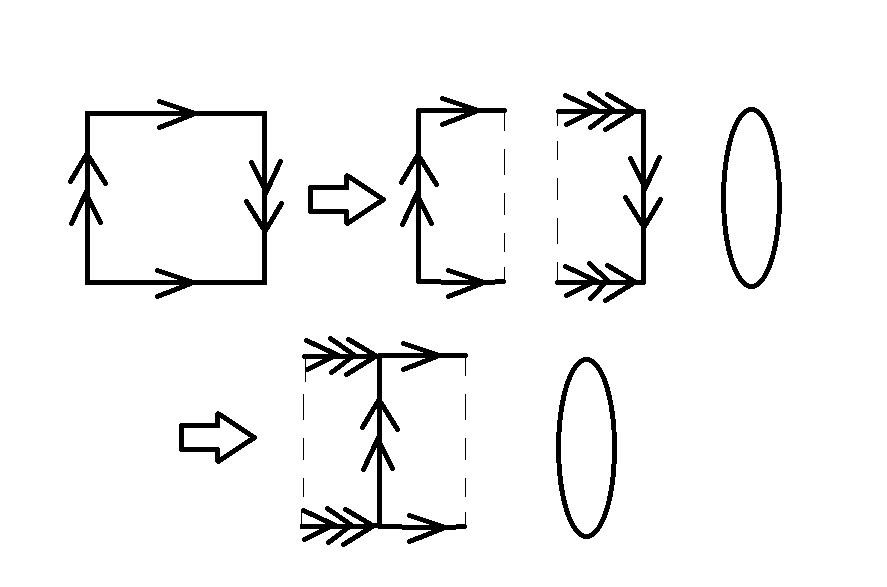} 
     \caption{Example~\ref{ej klein toro}.}
    \label{figura klein toro}
\end{figure}

Another example involving the Klein bottle $K$ is the following.

\begin{example}\label{ej fibrado klein}
    Consider the standard fiber bundle $S^1\rightarrow K\rightarrow S^1$. Let $C$ be the cone of the circle $S^1$ and let $X$ be the result of gluing $C$ and $K$ identifying the base of the cone with a section $s$ of the bundle. Let $f:X\rightarrow X$ be a homeomorphism such that $f(C)= C$ and $f(K)= K$  (and hence $f(s)=s$, since $s\equiv C\cap K$). Then we have
    \begin{align}\label{ec par ref penultimo ej}
        &\varLambda(f,X)=\varLambda(f,C)_X+\varLambda(f,K\setminus s)_X=\\
        &\varLambda(f_{|C},C)+\varLambda(f,K\setminus s)_X=1+\varLambda(f,K\setminus s)_X\;.\notag
    \end{align}
    But $K\setminus s$ is homeomorphic to an open Möbius strip $M'$. If, $f_{|K\setminus s}$ is good enough, we can extend it to a homeomorphism $\tilde{f}$ on the closed Moesbius strip $M$ and then the topological invariance of the combinatorial Lefschetz number will lead to 
    \begin{equation*}
        \varLambda(f,K\setminus s)_X=\varLambda(\tilde{f},M')_M\;.   
    \end{equation*}
    But, from the additivity of the combinatorial Lefschetz number we have
    \begin{equation*}
        \varLambda(\tilde{f},M')_M=\varLambda(\tilde{f},M)-\varLambda(\tilde{f},S)_M=\varLambda(\tilde{f},M)-\varLambda(\tilde{f}_{|S},S)\;,
    \end{equation*}
    where $S$ is the circle corresponding to the boundary of $M'$ in $M$. So we have reduced the problem of computing the Lefschetz number in a Klein bottle to a computation in a Möbius strip.

    If $f_{|K\setminus s}$ is not good enough, we can extend it to a homeomorphism $\tilde{f}$ of the one-point compactification of $K\setminus s$, that is, to the projective space. In this case, if $p$ is the point of the compactification, the topological invariance of the topological Lefschetz number implies $\varLambda(f,K\setminus s)_X=\varLambda(\tilde{f},\mathbb{R}P^2\setminus \{p\})_{\mathbb{R}P^2}$. But, by the additivity, we have
    
    \begin{equation*}
        \varLambda(\tilde{f},\mathbb{R}P^2\setminus \{p\})_{\mathbb{R}P^2}=\varLambda(\tilde{f},\mathbb{R}P^2)-\varLambda(\tilde{f},\{p\})_{\mathbb{R}P^2}=\varLambda(\tilde{f},\mathbb{R}P^2)-1\;,
    \end{equation*}
    and then, Equation~\eqref{ec par ref penultimo ej} becomes $\varLambda(f,X)=\varLambda(\tilde{f},\mathbb{R}P^2)$.
   
    Note however that the last equality is nothing new since $X$ is homotopic to $\mathbb{R}P^2$.
    
    In conclusion, if $f$ is good enough, we reduce the problem of computing the Lefschetz number on $X$ to a computation of the Lefschetz number on a Möbius strip and on a circle.
\end{example}
\begin{remark}
   With the notation of the previous example, given a map $f:X\rightarrow X$, we can use that the inclusions $K\rightarrow X$ and $C_0\rightarrow X$ ($C_0$ is the upper half of the cone $C$
   ) are cofibrations \cite[Sec.~5.3.1]{Dieck}. This may help to find a homeomorphism $f'$ homotopic to $f$ satisfying the hypothesis needed in Example~\ref{ej fibrado klein}.
\end{remark}
\begin{remark}
    Note that the combinatorial Lefschetz number can also be used as follows. Since most of the Lefschetz fixed-point theory is made for finite simplicial complexes, and hence the Lefschetz numbers can be obtained using computational methods, it is also relevant to improve these methods. In this direction, sometimes the combinatorial Lefschetz number will allow to discard most of the values outside a very small finite set of integers. A naive, but illustrative example of this is when, in the first step of Idea~\ref{computo en esferas con combinatorio}, we restrict the possible values of the Lefschetz number of a map in the circle to $0$ or $2$ (this result was already known, but we want to highlight in this remark the process how the combinatorial Lefschetz numbers allows to discard many options). 
    
    Note that the possibility of limiting the potential values of the Lefschetz number to a small set of integers can become very useful since, most of the times, the only thing we need to know is whether the Lefschetz number is different from zero. 
\end{remark}
\section{Fixed-Point Theorems in unbounded spaces}\label{seccion aplicaciones practicas}

Having a fixed-point theorem in non-compact spaces is one of the most important challenges of the topological fixed-point theory \cite[Problem 3]{Brown2}. From \cite{Leray}, where Leray introduced some ideas that can generalize the Lefschetz fixed-point theorem to non-compact spaces in some cases, several works have established new results on spaces that are not compact or even not bounded. Some examples of this are \cite{Browder2,E-F,Granas,Nussbaum1,Nussbaum2,Nussbaum3,Tromba}. Also, in the context of infinite dimensional Bannach spaces, we can mention \cite{Browder1}. Nevertheless, during the XXI century, as far as we know there are hardly any works in this line \cite{Cauty2,Cauty,Hochs}. However, except for \cite{Hochs} (where the study is restricted to isometries), in all the other articles that have appeared after \cite{Leray}, some compactness hypothesis is always needed. In this way, all of these works require the maps to have a certain compactness (to be compact, CAC, condensing,\dots). A good introduction to these results is \cite{Gorniewicz}. 

In our case, the second author and David Mosquera-Lois obtained in \cite{M-M1} a first fixed-point theorem for non-compact spaces based on the combinatorial Lefschetz number. Theorem~\ref{teor punto fijo} is a generalization of this theorem. However, in \cite{M-M1} and in Theorem~\ref{teor punto fijo}, the spaces, even though they can be non-compact, they are needed to be bounded.

Now, we will present some new results where the combinatorial Lefschetz number 
and some bounds of the fixed-point index allow us to have new fixed-point theorems for certain open maps in unbounded spaces (for simplicity, since the idea will be the same in both cases, we are going to state the theorems in terms of homeomorphisms, but, usually, it will be possible to extend the results to open maps). In addition to establishing a new method for detecting fixed points in unbounded spaces, this can be seen as one of the first methods where compactness is almost not needed (the only condition is to require that the maps have an extension to any compactification in $\mathbb{R}^n$, for which requiring them to be proper is enough). 
Note also that the restriction to proper maps or even to homeomorphisms is something usual in fixed-point theory (see, for example, \cite{Kelly 2,Lecalvez acotacion,Lecalvez1,Lecalvez2,LC-Y,Le Roux,Le Roux2}). Moreover, even if Idea~\ref{idea teorema final} can also be used with spaces of dimension greater than $2$ when the fixed-point index is bounded, the restriction to surface-like spaces in the new theorems of this section, is usual in fixed-point theory (the same examples above prove it).

Finally, in Section~\ref{subseccion nuevas acotaciones}, we will focus on the relevance of finding new bounds for the fixed-point index. Theorem~\ref{pseudoteorema} will emphasize the power of a conjecture like Conjecture~\ref{conjetura} if we could show that it is true. Indeed, we will also see, always due to Idea~\ref{idea teorema final}, that the combinatorial Lefschetz number can be used to find counterexamples of spaces where the fixed-point index is not bounded by a given quantity. 
We must say that the search for bounds for the fixed-point index is especially interesting in mathematics as, for example, \cite{Lecalvez acotacion,Le Roux} show. Several results have been obtained in this line. In addition to \cite{Lecalvez acotacion}, we can mention \cite{G-K,G-K2,Jiang acotacion,J-W-Z,Kelly,Kelly3,Z-Z}.
\subsection{Some new fixed-point theorems on unbounded spaces}
Let us start now by defining the combinatorial Lefschetz number for unbounded spaces, which indeed, by itself, is something quite interesting. This definition will always exist for closed definable subspaces of $\mathbb{R}^n$ or for spaces that are homeomorphic to such spaces, as we mention in Remark~\ref{obs vale para homeomorfos a definibles}. However, even if this family is much larger than the class of closed definable spaces, for the sake of simplicity, we will continue to state them in terms of definable spaces and not of spaces homeomorphic to a definable spaces.

\begin{definition}\label{def numero comb en no acotados}
    Let $U\subset \mathbb{R}^n$ be a definable space and let $f:U\rightarrow U$ be a continuous map. If there exists a definable compactification $X$ of $U$ (this means that $X$ and the inclusion of $U$ in $X$ are definable) and an open extension $f':X\rightarrow X$ of $f$ such that $f'(X\setminus U)\subset X\setminus U$, then we define the \textit{combinatorial Lefschetz number} of $f$ in $U$ as
    \begin{equation*}
        \varLambda_{\mathrm{comb}}(f,U)=\varLambda(f',U)_X\;.
    \end{equation*}
    Remark~\ref{obs numero para abtas es inv} implies that this number is well-defined.
\end{definition}

\begin{remark}
        This new combinatorial number is a generalization of that in \cite{M-M1} since, if $U\subset X$ is a definable subspace of a simplicial complex $X$ and $f:X\rightarrow X$ is an open map such that $f(U)\subset U$ and $f(X\setminus U)\subset X\setminus U$, we have
    \begin{equation*}
        \varLambda(f,U)_X=\varLambda(f,U)_{\overline{U}}=\varLambda_{\mathrm{comb}}(f_{|U},U)\;.
    \end{equation*}
\end{remark}

\begin{remark}\label{la extension de homeos de cerrados existe}
Now, if $U$ is definable and closed in $\mathbb{R}^n$, and it is also unbounded, the Alexandroff's compactification of $U$ can be viewed as the composition of the inclusion of $U$ in $\mathbb{R}^n$ with the compactification $c:\mathbb{R}^n\rightarrow S^n$, given by the inverse of the stereographic projection. Now, the inclusion is always definable and, by \cite{Shaviv}, $c$ is also definable. Since the composition of definable maps is also definable \cite[Lemma 1.2.3]{Dries}, the Alexandroff's compactification of a closed definable unbounded space of $\mathbb{R}^n$ is also definable. During the rest of the paper, $U^{\infty}$ will denote the Alexandroff's compactification of $U$. Indeed, if $f:U\rightarrow U$ is a homeomorphism, the map $\tilde{f}:U^{\infty}\rightarrow U^{\infty}$, defined as $f$ in $U$ and as $\tilde{f}(\infty)=\infty$, is also a homeomorphism (we will also keep this notation during the rest of the paper). Consequently, the combinatorial Lefschetz number can always be defined for the homeomorphisms of a definable space closed in $\mathbb{R}^n$. 
\end{remark}
\begin{remark}
    In general, if the map $f:U\rightarrow U$ is proper, it is possible to extend it to a map $\tilde{f}:U^{\infty}\rightarrow U^{\infty}$, but we must be careful since we need this extension to be an open map. So, even if the theorems of this section can be applied to some open maps too, in order to simplify the statements, we will work with homeomorphisms. Nevertheless, Idea~\ref{idea teorema final} has a straightforward generalization to open maps.
\end{remark}
From Theorem~\ref{teor punto fijo}, we obtain the following direct but relevant consequence about the fixed points of map extensions.
\begin{theorem}\label{teor pto fijo de la extension}
 Let $U$ be a definable and closed subspace of $\mathbb{R}^n$, and $f:U\rightarrow U$ a continuous map without fixed points and such that $\varLambda_{\mathrm{comb}}(f,U)\neq 0$. Then, every open map $f':X\rightarrow X$ that extends $f$ to a definable compactification $X$ of $U$ with $f(X\setminus U)\subset X\setminus U$ must have a fixed point in $X\setminus U$.
\end{theorem}

In this section, the main tools we are going to use are the additivity and the topological invariance of the combinatorial Lefschetz number, and the bounds of the fixed-point index. Both the index of an isolated fixed point and the index of a fixed-point class admit sometimes some bounds. Among all the results providing bounds for the index, we will use the four below. 

When $x$ is an isolated fixed point of $f$, $\mathrm{ind}(x)$ will mean the index $i(X,f,U)$, where $U$ is an open subspace of $X$ such that $x\in U$ is the only fixed point of $f$ in $\overline{U}$ (the additivity axiom of the index \cite[Sec.~IV.A]{Brown} implies that this index is independent of the chosen $U$ satisfying this condition).

The next theorem appears in \cite{Jiang acotacion}. A similar result is in \cite{Lecalvez acotacion}.
\begin{theorem}\label{acotacion superficies}
    Let $X$ be a compact and connected surface of negative Euler characteristic and let $f:X\rightarrow X$ be a continuous map. In this case, the index of all fixed-point classes $F$ of $f$ is less or equal to $1$, and we have 
        \begin{equation*}
            \sum_{\mathrm{ind}(f,F)<-1}(\mathrm{ind}(f,F)+1)\geq2\chi(X)\;.
        \end{equation*}
    In particular we have
    \begin{equation*}
        |\varLambda(f,X)-\chi(X)|\leq N(f)-\chi(X)\;.
    \end{equation*}
\end{theorem}

\begin{theorem}[{\cite{Kelly}}]\label{acotacion superficies borde}
    Let $X$ be a compact connected surface with boundary and let $f:X\rightarrow X$ be a continuous map such that $\#\mathrm{Fix}(f)\leq\#\mathrm{Fix}(g)$ for every continuous map $g$ homotopic to $f$ {\rm(}here $\mathrm{Fix}(f)$ is the set of fixed points of $f${\rm)}. In this case, $\mathrm{ind}(x)\le 1$ for all $x\in\mathrm{Fix}(f)$, and 
    \begin{equation*}
                \sum_{\mathrm{ind}(x)< -1} (\mathrm{ind}(x))+1)\geq 2\chi(X)\;.
    \end{equation*}
\end{theorem}

\begin{theorem}[{\cite{G-K}}]\label{acotacion wedge}
    Let $X$ be the wedge product of $m$ compact connected surfaces $X_1,\ldots,X_m$ {\rm(}some of them may have boundary{\rm)} with non positive Euler characteristic. Let $f:X\rightarrow X$ be a map such that $\#\mathrm{Fix}(f)\leq\#\mathrm{Fix}(g)$ for every continuous map $g:X\rightarrow X$ homotopic to $f$. Then:
    \begin{equation*}
        \sum_{\mathrm{ind}(x)\geq 1} (\mathrm{ind}(x))-1)\leq 0\;,\quad
        \sum_{\mathrm{ind}(x)\leq -1} (\mathrm{ind}(x))+1)\geq 2\chi(X)\;.
    \end{equation*}
    In particular,  $\mathrm{ind}(x)\le 1$ for all $x\in\mathrm{Fix}(f)$.
\end{theorem}

\begin{theorem}[{Cf.\cite[Theorem 1.1]{J-W-Z}}]\label{acotacion en grafos}
    Let $X$ be a finite connected graph and let $f:X\rightarrow X$ be a continuous map. Then the index of each fixed-point class of $f$ is less or equal to $1$.
\end{theorem}
 
 As in Theorem~\ref{acotacion superficies}, the bounds of the index lead to
     \begin{equation*}
        |\varLambda(f,X)-\chi(X)|\leq N(f)-\chi(X)\;.
    \end{equation*}
Indeed, these bounds also allow us to obtain certain fixed-point results for unbounded spaces. As a first idea of what we are going to do with the combinatorial Lefschetz number, imagine that $U$ is an unbounded space, $f:U\rightarrow U$ a continuous map and suppose that $f$ can be extended to a map $\tilde{f}:U^\infty\rightarrow U^\infty$ (recall that $U^\infty$ means the Alexandroff's compactification of $U$). Imagine also that $U^\infty$ and $\tilde{f}$ are in the conditions of Theorem~\ref{acotacion superficies}, \ref{acotacion superficies borde}, \ref{acotacion wedge} or \ref{acotacion en grafos}, so the index of any fixed point is bounded by $1$. In this case, if $\varLambda(\tilde{f},U^\infty)>1$, $f$ must have a fixed point. This is true because, if $f$ has no fixed points, $\tilde{f}$ can have at most one fixed point (the point at infinity of the Alexandroff's compactification) whose index is bounded by one. In this case,
\begin{equation*}
    \varLambda(\tilde{f},U^\infty)= \mathrm{ind}(\infty)\leq1\;,
\end{equation*}
where the first equality follows from the normalization and additivity axioms of the index. But this contradicts the assumption $\varLambda(\tilde{f},U^\infty)>1$.

This first and primitive idea, although quite simple and can be considered as a corollary of the formula
\begin{equation*}
    |\varLambda(f,X)-\chi(X)|\leq N(f)-\chi(X)\;,
\end{equation*}
has not been exploited in fixed-point theory. In particular, taking compactifications to study the existence of fixed points with the Lefschetz number is not very common because the first thing one thinks is that the the fixed points detected by the Lefschetz number are those in the corona of the compactification.

Indeed, and this is the key step of this section, the importance of our results is that, thanks to the combinatorial Lefschetz number, we are able to state our hypothesis on $U$ instead of $U^\infty$ (similar results in $U$ for the Lefschetz number would not be true). Moreover, even though, by the additivity of the combinatorial Lefschetz number, $\varLambda_{\mathrm{comb}}(f,U)$ and $\varLambda(\tilde{f},U^\infty)$ are related, $\varLambda_{\mathrm{comb}}(f,U)$ is, in some strong way, independent of $\varLambda(\tilde{f},U^\infty)$ since, due to the topological invariance of the combinatorial Lefschetz number, this number is independent of the chosen compactification and then we are free to use the most appropriate one to compute it. We have seen along this paper that this may allow to compute $\varLambda_{\mathrm{comb}}(f,U)$ very easily. This is very strong since, in particular, $\varLambda_{\mathrm{comb}}(f,U)$ will be always easier to compute than $\varLambda(\tilde{f},U^\infty)$. The examples in this section will illustrate the relevance of this advantage, especially Example~\ref{ej compactificaciones al reves}.

So, using the combinatorial Lefschetz number, the main idea of this section can be summarized as follows:

\begin{idea}\label{idea teorema final}
    Let $U\subset \mathbb{R}^n$ be a definable subspace closed in $\mathbb{R}^n$ and let $f:U\rightarrow U$ be a homeomorphism. Suppose that $U^\infty$ and $\tilde{f}$ satisfy the hypothesis of Theorem~\ref{acotacion superficies}, \ref{acotacion superficies borde}, \ref{acotacion wedge} or \ref{acotacion en grafos}.

    Assume now that $f$ has no fixed points. In this case, $\infty$ is the only fixed point of $\tilde{f}$ (in particular, the index of the point $\infty$ is equal to the one of its fixed-point class). Then, by Theorem~\ref{acotacion superficies}, \ref{acotacion superficies borde}, \ref{acotacion wedge} or \ref{acotacion en grafos}, we have that $\mathrm{ind}(\infty)\le 1$ (in fact, we will also have a lower bound depending on the Euler characteristic of the space).

    We also have

    \begin{align*}
        \mathrm{ind}(\infty)=\varLambda(\tilde{f},U^\infty)&=\varLambda(\tilde{f},U)_{U^\infty}+\varLambda(\tilde{f},\{\infty\})_{U^\infty}\\
        &=\varLambda_{\mathrm{comb}}(f,U)+1\;,
    \end{align*}
    where the first equality is a consequence of the normalization and additivity axioms of the index. Then,
    \begin{equation*}
        \varLambda_{\mathrm{comb}}(f,U)=\mathrm{ind}(\infty)-1\;.
    \end{equation*}
    Since $\mathrm{ind}(\infty)\leq 1$, we get $\varLambda_{\mathrm{comb}}(f,U)\le 0$. So, when $\varLambda_{\mathrm{comb}}(f,U)\geq 1$, $f$ must have a fixed point (and in this way we can state some fixed-point theorems for unbounded spaces). In order to simplify the notation, most of the results that we are going to state use only the upper bounds of the index. However, we must emphasize that, using the same ideas, more powerful results can be obtained combining the upper and lower bounds of Theorems~\ref{acotacion superficies}, \ref{acotacion superficies borde}, \ref{acotacion wedge} and \ref{acotacion en grafos}, but, since these theorems depend on the Euler characteristic of the compactification (or the combinatorial Euler characteristic of the original space), in this paper, most of the times we will not use the lower bounds to avoid tedious statements. Only in Theorem~\ref{primer resultado gordo, wedge}, we will illustrate the use of these lower bounds. 
\end{idea}
Now, from Theorems~\ref{acotacion superficies}, \ref{acotacion superficies borde}, \ref{acotacion wedge} and \ref{acotacion en grafos}, we obtain four fixed-point theorems. The proofs are similar in all of them, so we will only prove Theorem~\ref{primer resultado gordo, wedge}. Note that there are no exist similar results for the usual Lefschetz number.

\begin{theorem}\label{primer resultado gordo, wedge}
    Let $U_1,\ldots, U_m\subset\mathbb{R}^n$ be disjoint definable and closed subspaces of $\mathbb{R}^n$ such that $U_i^\infty$ is a compact and connected surface {\rm(}with or without boundary{\rm)} for every $i\in\{1,\ldots,m\}$. Suppose also that $\varLambda_{\mathrm{comb}}(\mathrm{id},U_i)\leq-1$ {\rm(}that is, the combinatorial Euler characteristic  of the $U_i$'s is negative{\rm)}. Finally, let $f:\bigcup_{i=1}^m U_i\rightarrow \bigcup_{i=1}^m U_i$ be a homeomorphism. If, either $\varLambda_{\mathrm{comb}}(f,\bigcup_{i=1}^m U_i)>0$, or $\sum_i\varLambda_{\mathrm{comb}}(\mathrm{id},U_i)=-1$ and $\varLambda_{\mathrm{comb}}(f,\bigcup_{i=1}^mU_i)<-2$, $f$ must have a fixed point.
\end{theorem}
\begin{proof}
   Suppose that $f$ has no fixed points and let us define $U=\bigcup_{i=1}^mU_i$. In this case, $\infty$ is the only fixed point of $\tilde{f}$.
   
    On the one hand, for every $i\in\{1,\ldots,m\}$, we have
    \begin{equation*}
        \chi(U_i^\infty)=\varLambda(\mathrm{id},U_i^\infty)=\varLambda_{\mathrm{comb}}(\mathrm{id},U_i)+\varLambda(\mathrm{id},\infty)=\varLambda_{\mathrm{comb}}(\mathrm{id},U_i)+1\;.
    \end{equation*}
    Since $\varLambda_{\mathrm{comb}}(\mathrm{id},U_i)\leq -1$, we see that $U_i^\infty$ is a compact connected surface of non-positive Euler characteristic. Indeed, since
    \begin{equation*}
        \varLambda(\tilde{f},U^\infty)=\varLambda_{\mathrm{comb}}(f,U)+1
    \end{equation*}
    and $\varLambda_{\mathrm{comb}}(f,U)>0$, we obtain $\varLambda(\tilde{f},U^\infty)\neq 0$, and therefore the classical Lefschetz fixed-point theorem implies the existence of a fixed point of every map homotopic to $\tilde{f}$. But, since we were assuming that $\tilde{f}$ has only one fixed point, we are under the conditions of Theorem~\ref{acotacion wedge}, and so $\mathrm{ind}(\infty)\leq 1$.

    In this case,
    \begin{equation*}
        1\geq \mathrm{ind}(\infty)=\varLambda(\tilde{f},U^\infty)=\varLambda_{\mathrm{comb}},(f,U)+1
    \end{equation*}
    and finally $\varLambda_{\mathrm{comb}}(f,U)\leq0$, which is a contradiction. In consequence, $f$ must have a fixed point.

    For the second part of the theorem, suppose again that $f$ has no fixed points. Note that $\sum_i\varLambda_{\mathrm{comb}}(\mathrm{id},U_i)=-1$ implies $\chi(U^\infty)=-1+1=0$. Indeed, since $\varLambda_{\mathrm{comb}}(f,U)<-2$, we deduce $\varLambda(\tilde{f},U^\infty)<-2+1=-1$ and, in particular, $\varLambda(\tilde{f},U^\infty)\neq 0$. Hence, every map homotopic to $\tilde{f}$ must have a fixed point and, since we were assuming that $\tilde{f}$ has only one fixed point, we are under the conditions of Theorem~\ref{acotacion wedge}. Hence, we obtain $\mathrm{ind}(\infty)\geq -1$. But we also have
    \begin{align*}
        -1\leq\mathrm{ind}(\infty)=\varLambda(\tilde{f},U^\infty)=\varLambda_{\mathrm{comb}}(f,U)+1\;,
    \end{align*}
    and, consequently, $\varLambda_{\mathrm{comb}}(f,U)\geq -2$, which is another contradiction.
\end{proof}

\begin{remark}
    Theorem~\ref{acotacion wedge} also gives lower bounds of the index in the case where $\chi(U^\infty)<0$. It is also possible to use these bounds with Idea~\ref{idea teorema final} to obtain an analogous result to the second part of Theorem~\ref{primer resultado gordo, wedge}. 
\end{remark}

In order to handle this new tool, let us begin with a first and theoretical example. Later, in Example~\ref{ej compactificaciones al reves}, we will present a situation that is, although more complicated than this one, very illustrative and interesting.

\begin{example}\label{wedge dos klein}
    Let $U$ be the space in Figure~\ref{figura klein infinito}. It consists of two different Klein bottles minus one point, and then we have pulled a neighborhood of each of these points towards the infinity. It is easy to see that $U^\infty$ is the wedge product of two Klein bottles and hence, for homeomorphisms $f:U\rightarrow U$, we can use Theorem~\ref{primer resultado gordo, wedge} to detect fixed points. Indeed, if $f$ is not too bad (this will mean that the extension $\tilde{f}'$ below exists), it will be possible to compute $\varLambda_{\mathrm{comb}}(f,U)$ as $\varLambda(\tilde{f}',U)_{4\mathbb{P}}$ instead of $\varLambda(\tilde{f},U^\infty)$, where $4\mathbb{P}$ is the nonorientable surface of genus $4$ that results of the compactification of $U$ described in Figure~\ref{segunda compactificacion klein infinito} and $\tilde{f}':4\mathbb{P}\rightarrow4\mathbb{P}$ is a homeomorphism that extends $f$.
\end{example}
     \begin{figure}[htb]
    \centering
     \includegraphics[scale=0.35]{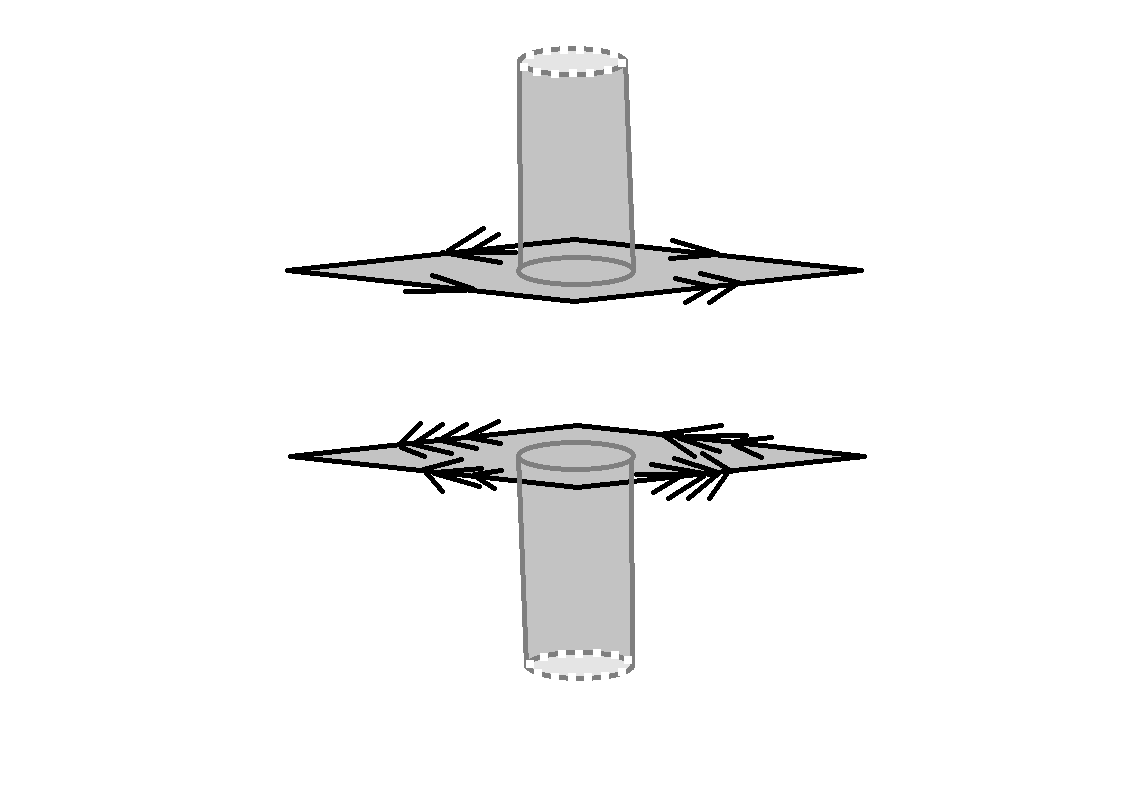} 
     \caption{Space $U$.}
     \label{figura klein infinito}
\end{figure}
     \begin{figure}[htb]
    \centering
     \includegraphics[scale=0.35]{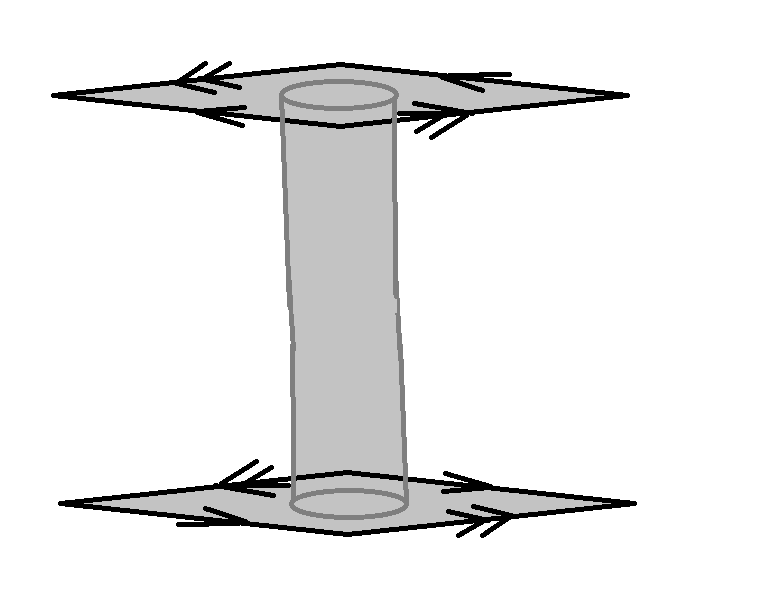} 
     \caption{A different compactification: the surface $4\mathbb{P}$.}
     \label{segunda compactificacion klein infinito}
\end{figure}

\begin{remark}
    The number of possible examples of this kind is very large. Even without increasing the number of surfaces in the wedge, we can consider a similar space to that in Example~\ref{wedge dos klein} but with a Klein bottle and a torus instead of two Klein bottles. In this case, another compactification different from Alexandroff's one can be, once again, a nonorientable surface. 
\end{remark}

Theorem~\ref{acotacion en grafos} allows us to obtain the following result, whose proof is similar to that of the first statement in Theorem~\ref{primer resultado gordo, wedge}.

\begin{theorem}\label{primer resultado gordo, grafos}
    Let $U\subset \mathbb{R}^n$ be a definable connected subspace that is closed in $\mathbb{R}^n$ and of dimension $1$ {\rm(}by this we mean that, using Theorem~\ref{thm:triangulation}, we can triangulate it by a complex of dimension $1${\rm)}. Let $f:U\rightarrow U$ be a homeomorphism such that $\varLambda_{\mathrm{comb}}(f,U)\geq 1$. In this case, $f$ must have a fixed point.
\end{theorem}

 Theorem~\ref{acotacion superficies borde} allows us to obtain the following result.
\begin{theorem}\label{primer resultado gordo, super con borde}
    Let $U\subset\mathbb{R}^n$ be a definable closed subspace of $\mathbb{R}^n$ such that $U^\infty$ is a compact connected surface with boundary. Let $f:U\rightarrow U$ be a homeomorphism such that $\varLambda_{\mathrm{comb}}(f,U)\geq 1$. Then, $f$ must have a fixed point.
\end{theorem}
Theorem~\ref{acotacion superficies} allows us to obtain the following result.
\begin{theorem}\label{primer resultado gordo, superficies}
    Let $U\subset\mathbb{R}^n$ be a definable closed subspace of $\mathbb{R}^n$ such that $U^\infty$ is a compact connected surface of negative Euler characteristic. Let $f:U\rightarrow U$ be a homeomorphism. Then, $\varLambda_{\mathrm{comb}}(f,U)\geq 1$ implies the existence of a fixed point of $f$.
\end{theorem}

This argument can be generalized. That is, we can use a bound of the index in some compactification of $U$ (not only in Alexandroff's compactification) where $f$ admits an extension, but compute $\varLambda_{\mathrm{comb}}(f,U)$ in another compactification where the computations are easier. The following example illustrates this idea.

\begin{example}\label{ej compactificaciones al reves}
    Let $U$ be the space in Figure~\ref{dos toros infinitos}. This space consists of the disjoint union of two tori from which we have removed a point and then we have pulled a neighborhood of this point towards the infinity. 
    
    We can state the following assertion:
\begin{claim}
    Let $f:U\rightarrow U$ be a homeomorphism that can be extended to a homeomorphism $f'$ of the compactification of $U$ given by the double torus $2\mathbb{T}$ (see Figure~\ref{doble toro}). Suppose also that $f'$ fixes the circle that we add to $U$ in order to obtain the double torus. Then, if $\varLambda_{\mathrm{comb}}(f,U)\geq 2$, $f$ must have a fixed point.
\end{claim}
    The proof is similar to the ones of the previous theorems. To begin with, we suppose that $f$ has no fixed points, and so the circle that we add to $U$ in order to obtain $2\mathbb{T}$ is the set of fixed points of $f'$. Since this set is connected, it is the only fixed-point class of $f$ \cite[Corollary~1.14]{Jiang}. Let $F$ be this fixed-point class. Now, Theorem~\ref{acotacion superficies} implies $\mathrm{ind}(F)\leq 1$. Hence, since the definition of $\varLambda_{\mathrm{comb}}(f,U)$ is independent of the compactification of $U$,
   \begin{align*}
        1\geq\mathrm{ind}(F)=\varLambda(f',2\mathbb{T})&=\varLambda_{\mathrm{comb}}(f,U)+\varLambda(f',|F|)_{2\mathbb{T}}=\\
        \varLambda_{\mathrm{comb}}(f,U)+\varLambda(f'_{||F|},|F|)_{|F|}&=\varLambda_{\mathrm{comb}}(f,U)+\varLambda(\mathrm{id},|F|)=\\
        \varLambda_{\mathrm{comb}}(f,U)+\chi(S^1)&=\varLambda_{\mathrm{comb}}(f,U)+0\;,
   \end{align*}
   where $|F|$ represents the points of the class $F$; that is, the circle that we added to $U$ to obtain the double torus. Hence, we get $\varLambda_{\mathrm{comb}}(f,U)\leq 1$, a contradiction. 
    
   We will show now a specific case of this example where, in fact, we can verify the existence of fixed points with the naked eye. Let $g:U\rightarrow U$ be the rotation of $180$ degrees around the line in Figure~\ref{dos toros infinitos}. In a circle of the cylinders, the restriction of this map is a rotation $r:S^1\rightarrow S^1$ of $180$ degrees. We define now a homeomorphism $f$ homotopic to $g$ with the following description:

   In the region of $U$ that looks like a torus in Figure~\ref{dos toros infinitos}, we set $f=g$. However, in both cylinders, $f$ is the result of the canonical homotopy between $r:S^1\rightarrow S^1$ and $\mathrm{id}_{|S^1}$ such that the homotopy reaches $\mathrm{id}_{|S^1}$ at the infinity.

   We can extend the homeomorphism $f$ to a homeomorphism of $2\mathbb{T}$ that fixes $|F|$. Let $f'$ be this homeomorphism.

   Now, instead of computing $\varLambda_{\mathrm{comb}}(f,U)=\varLambda(f',U)_{2\mathbb{T}}$ (something that could become really tedious), we can compute $\varLambda_{\mathrm{comb}}(f,U)=\varLambda(f'',U)_{\mathbb{T}_1\sqcup\mathbb{T}_2}$, where $f''$ is the homeomorphism that maps each point $(x_1,y_1,x_2,y_2)\in\mathbb{T}_i=S^1_i\times S^1_i$ to $(-x_1,y_1,x_2,-y_2)\in\mathbb{T}_i$ for $i=1,2$ (note that $f''$ is not an extension of $f$ in $\mathbb{T}_1\sqcup\mathbb{T}_2$, but is homotopic to one of such extensions). Figure~\ref{dos toros} represents this homeomorphism.

   Due to the additivity of the combinatorial Lefschetz number, we have $$\varLambda(f'',U)_{\mathbb{T}_1\sqcup\mathbb{T}_2}=\varLambda(f'',\mathbb{T}_1\setminus \{p\})_{\mathbb{T}_1\sqcup\mathbb{T}_2}+\varLambda(f'',\mathbb{T}_2\setminus \{q\})_{\mathbb{T}_1\sqcup\mathbb{T}_2}$$
   and, by symmetry, 
   $$\varLambda(f'',U)_{\mathbb{T}_1\sqcup\mathbb{T}_2}=2\varLambda(f'',\mathbb{T}_1\setminus \{p\})_{\mathbb{T}_1\sqcup\mathbb{T}_2}=2\varLambda(f''_{|\mathbb{T}_1},\mathbb{T}_1\setminus \{p\})_{\mathbb{T}_1}.$$
   So, let us compute $\varLambda(f''_{|\mathbb{T}_1},\mathbb{T}_1\setminus \{p\})_{\mathbb{T}_1}$.

   Due to the additivity, $\varLambda(f''_{|\mathbb{T}_1},\mathbb{T}_1\setminus \{p\})_{\mathbb{T}_1}=\varLambda(f''_{\mathbb{T}_1},\mathbb{T}_1)-1$. Now, $f''^\ast:H^1(S^1_1\times S^1_1)\rightarrow H^1(S^1_1\times S^1_1)$ is the homomorphism that maps each generator $\alpha_i^\ast$ to $-\alpha_i^\ast$  for $i=1,2$, and hence, due to \cite{B-B-P-T}, 
   \begin{equation*}
       \varLambda(f''_{|\mathbb{T}_1},\mathbb{T}_1)=
             \mathrm{det} \begin{pmatrix}
          1-(-1)&0\\
          0&1-(-1)
        \end{pmatrix}
        =4\;,
   \end{equation*}
   and so $\varLambda(f''_{|\mathbb{T}_1},\mathbb{T}_1\setminus \{p\})_{\mathbb{T}_1}=3$. Consequently, $\varLambda_{\mathrm{comb}}(f,U)=2\cdot 3=6\geq 2$, obtaining that $f$ must have a fixed point.\qedhere
\end{example}
     \begin{figure}[htb]
    \centering
     \includegraphics[scale=0.4]{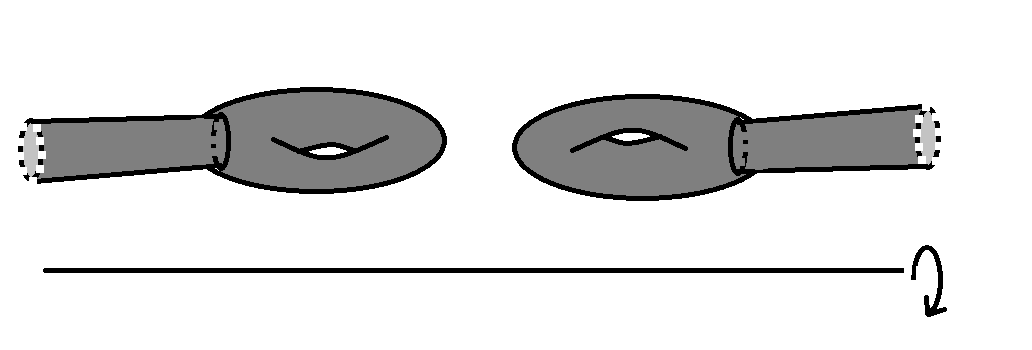} 
     \caption{The space $U$.}
     \label{dos toros infinitos}
\end{figure}
     \begin{figure}[htb]
    \centering
     \includegraphics[scale=0.4]{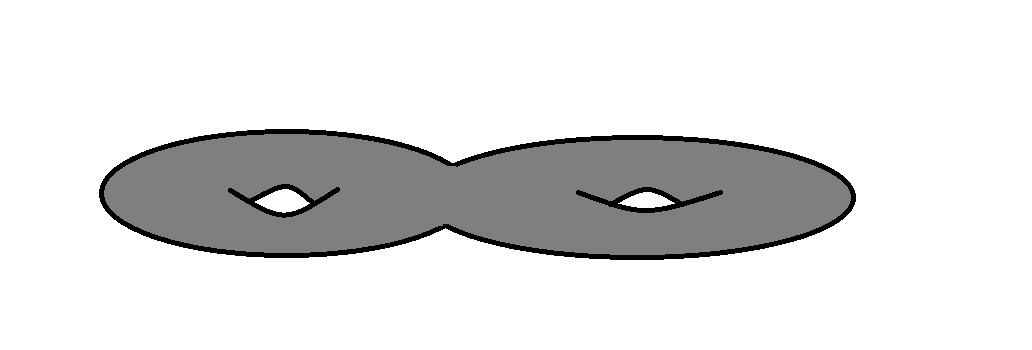} 
     \caption{Double torus $2\mathbb{T}$.}
     \label{doble toro}
\end{figure}
     \begin{figure}[htb]
    \centering
     \includegraphics[scale=0.4]{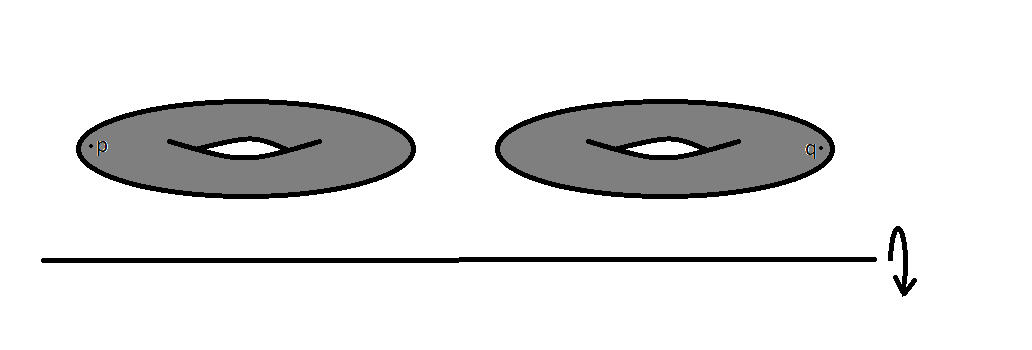} 
     \caption{$\mathbb{T}_1\sqcup\mathbb{T}_2$.}
     \label{dos toros}
\end{figure}

Before the last subsection of this paper, let us introduce a last example, in this case, of an application of Theorem~\ref{primer resultado gordo, wedge} for the case of just one surface in the wedge product. 
The purpose of this example is to show that the compactifications of the unbounded space can be much simpler than the Alexandroff's compactification, and hence, being able to compute $\varLambda_{\mathrm{comb}}$ is a great advantage compared to computing it only in the Alexandroff's compactification.
\begin{example}\label{klein punteada}
    Let $U'$ be one of the two components of the space $U$ in Example~\ref{wedge dos klein}. Then, $U'$ consists of a Klein bottle from which we have removed a point and then pulled a neighborhood of this point towards the infinity. Let $f:U\rightarrow U$ be a homeomorphism. The Alexandroff's compactification $U'^\infty$ of this space is the Klein bottle and so, Theorem~\ref{primer resultado gordo, wedge} gives us a criteria for detecting fixed points. However, if $f$ is good enough (in particular if $f$ is similar to the map of Example~\ref{ej compactificaciones al reves}), it is possible to compute $\varLambda_{\mathrm{comb}}(f,U')$ in the compactification of Figure~\ref{figura klein punteada}, which is much easier to work with.
\end{example}

     \begin{figure}[htb]
    \centering
     \includegraphics[scale=0.60]{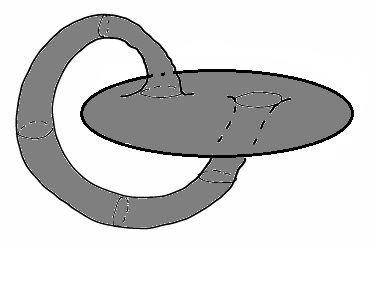} 
     \caption{Another compactification of $U$.}
     \label{figura klein punteada}
\end{figure}

\subsection{The importance of new bounds}\label{subseccion nuevas acotaciones}
This last subsection will focus on two different purposes. First, to show that the combinatorial Lefschetz number of unbounded spaces, together with Idea~\ref{idea teorema final}, can be used to determine when the index is not bounded by a given integer. We will illustrate the idea of using these tools in Theorem~\ref{el indice en la esfera no esta acotado}. 
Later, as the final part of this paper, we will insist on the importance of finding new bounds of the index. We will show their power in Theorem~\ref{pseudoteorema} and Example~\ref{ultimo ejemplo}.

The bounds of the fixed-point index are extensively studied in fixed-point theory. The relevance of these bounds can be seen in \cite{Jiang acotacion,Le Roux} or in the present paper. Consequently, many other works have been done providing new results on bounds for the index. We could mention, for example, \cite{G-K,G-K2,J-W-Z,Kelly,Kelly3}. Some works also focus on finding counterexamples where the index is not bounded \cite{Z-Z}.

It is precisely in this search of counterexamples that the combinatorial Lefschetz number of unbounded spaces and Idea~\ref{idea teorema final} become a very useful tool. In order to illustrate this, we will provide another proof, using Idea~\ref{idea teorema final}, of a well-known result that states that the fixed-point index in the sphere $S^2$ is not bounded by $1$. Again, we must note that a similar argument (although it is not frequent in the literature) can be considered using the classical Lefschetz number, but the advantage of the combinatorial Lefschetz number is that it is usually much easier to compute. In this theorem, for example, instead of using Alexandroff's compactification for the open square $(0,1)^2$, we use the closed square $[0,1]^2$ (in this case, both computations are similar; however, in a less illustrative context, the computation of the Lefschetz number in Alexandroff's compactification can become very tedious).

\begin{theorem}\label{el indice en la esfera no esta acotado}
    There exists a homeomorphism $\tilde{f}:S^2\rightarrow S^2$  such that $\tilde{f}$ has an isolated fixed point of index $2$ and $\#\mathrm{Fix}(\tilde{f})\leq\#\mathrm{Fix}(g)$ for every map $g$ homotopic to $\tilde{f}$. 
\end{theorem}
\begin{proof}
    Suppose that the index of every isolated fixed point is less or equal to $1$ for every map satisfying the hypothesis. In this case, repeating again the arguments of Idea~\ref{idea teorema final}, it is easy to see that every homeomorphism $f:\mathbb{R}^2\rightarrow \mathbb{R}^2$ with $\varLambda_{\mathrm{comb}}(f,\mathbb{R}^2)\geq 1$ must have a fixed point. 

    However, if we define $f:\mathbb{R}^2\rightarrow \mathbb{R}^2$ as the homeomorphism that maps each $(x,y)$ to $(x,y+1)$, $f$ has no fixed points but $\varLambda_{\mathrm{comb}}(f,\mathbb{R}^2)=1$. 
    This is because we can triangulate  $\mathbb{R}^2$ by $(0,1)^2$ is such a way that the homeomorphism $h$ induced by $f$ through this triangulation can be extended to a homeomorphism $h':[0,1]^2\rightarrow [0,1]^2$ such that $h'_{\partial[0,1]^2}$ is homotopic to the identity. Hence,
    \begin{align*}
        &\varLambda_{\mathrm{comb}}(f,\mathbb{R}^2):=\varLambda(h',(0,1)^2)_{[0,1]^2}=\\
        &\varLambda(h',[0,1]^2)_{[0,1]^2}-\varLambda(h',\partial [0,1]^2)_{[0,1]^2}=\\
        &\varLambda(h',[0,1]^2)-\varLambda(h',\partial [0,1]^2)=1-\varLambda(\mathrm{id},\partial[0,1]^2)=1.
    \end{align*}
    (Here $\varLambda(h',[0,1]^2)=1$ since the space is simply connected.)
\end{proof}

Now we will see the relevance of finding new bounds for the index. A very interesting situation is when our space is the collapse of certain compact connected surfaces of non-positive Euler characteristic in such a way that, except at a finite number of points where the resulting space is locally a wedge product of surfaces, at all other points, the collapse is locally a surface. Figure~\ref{toro colapso circunferencia} show two possible examples of these spaces, both obtained from a collapse of the torus. Especially for $X_2$, although the point of the collapse is not a global separating point, it seems some argument similar to that used in \cite{G-K} to prove Theorem~\ref{acotacion wedge} could give a bound for the index.

   \begin{figure}[htb]
    \centering
     \includegraphics[scale=0.55]{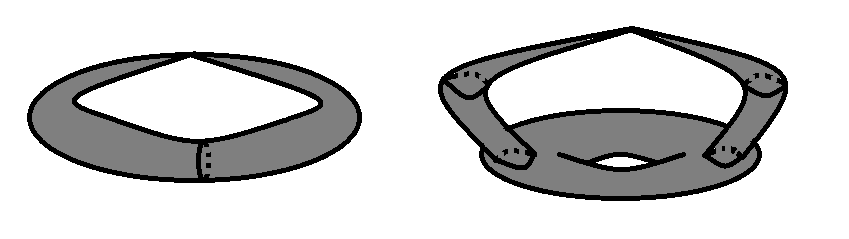} 
     \caption{Collapse of the torus by a circle ($X_1$), and collapse of the torus by two points ($X_2$).}
     \label{toro colapso circunferencia}
\end{figure}

The third space in Figure~\ref{pasos toro} shows another example of these spaces. In this case, $U^\infty/S$ is locally a surface at all points except two. 

Finally, another example can be the wedge of surfaces by more than one point (see the first space in Figure~\ref{wedge por varios puntos}). These spaces are especially interesting when using the combinatorial number, since, to compute $\varLambda_{\mathrm{comb}}(f,U)$, we can find more regular compactifications than $U^\infty$ where the computations of the combinatorial Lefschetz number would be much easier. The second space in Figure~\ref{wedge por varios puntos} shows this kind of more regular compactification.

     \begin{figure}[htb]
    \centering
     \includegraphics[scale=0.3]{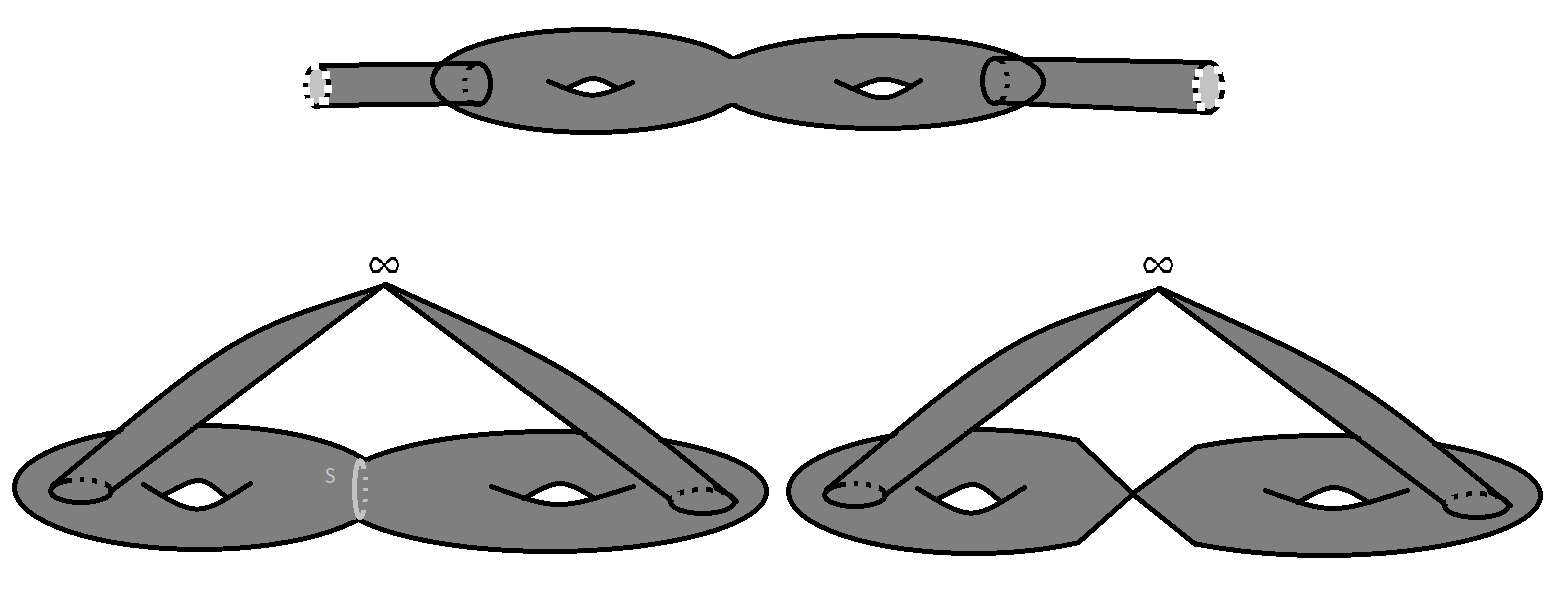} 
     \caption{$U$, $U^\infty$ and $U^\infty/S$.}
     \label{pasos toro}
\end{figure}

     \begin{figure}[htb]
    \centering
     \includegraphics[scale=0.35]{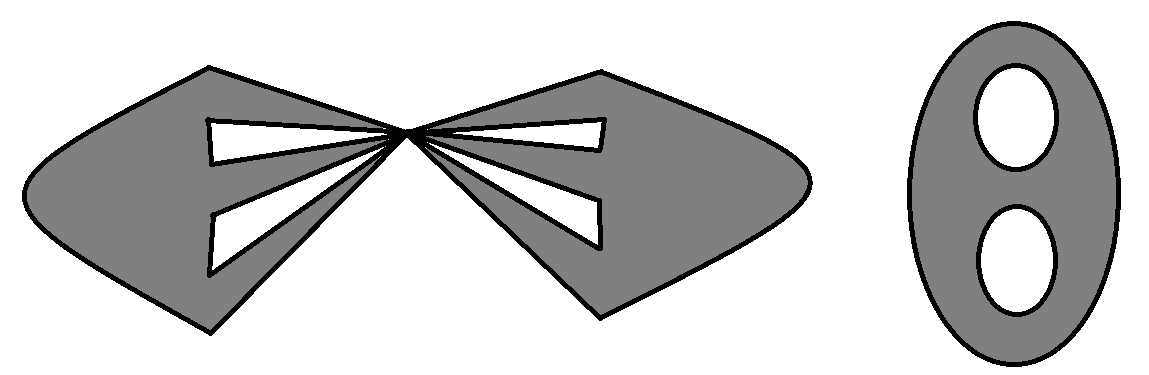} 
     \caption{Two different compactifications of the same space.}
     \label{wedge por varios puntos}
\end{figure}

Imagine that, at least for some of these types of collapse, the index is bounded. We can propose the following conjecture. Then we will show the applications that this conjecture would have if it were true.
\begin{conjetura}\label{conjetura}
    Let $X$ be a collapse of a finite number of surfaces in one of the ways described above, let $f:X\rightarrow X$ be a homeomorphism such that $\#\mathrm{Fix}(f)\leq\#\mathrm{Fix}(g)$ for every map $g$ homotopic to $f$ and let $x\in X$ be a fixed point of $f$. Then, $\mathrm{ind}(x)\leq 1$ and the fixed-point index of each isolated fixed point is also bounded by a lower bound depending on de Euler characteristic of $X$.
\end{conjetura}

If this conjecture is true, we can obtain the following result, whose proof is similar to that of the first statement in Theorem~\ref{primer resultado gordo, wedge}.
\begin{theorem}\label{pseudoteorema}
    Suppose that Conjecture~\ref{conjetura} is true. Let $U\subset\mathbb{R}^n$ be a closed definable subspace of $\mathbb{R}^n$ such that $U^\infty$ is one of the spaces for which Conjecture~\ref{conjetura} applies and let $f:U\rightarrow U$ be a homeomorphism. Then, if $\varLambda_{\mathrm{comb}}(f,U)\geq1$, $f$ must have a fixed point.
\end{theorem}
Let us illustrate the advantage this theorem with a last example. The idea and the computations of this example are similar to those in Example~\ref{ej compactificaciones al reves}. Again, in this example we can verify the existence of fixed points with the naked eye.

\begin{example}\label{ultimo ejemplo}
    Let $U$ be the space in Figure~\ref{toro infinito doble} and let $f$, as in Example~\ref{ej compactificaciones al reves}, be the rotation of $180$ degrees around the plane $z=0$. Then, as in Example~\ref{ej compactificaciones al reves}, computing $\varLambda_{\mathrm{comb}}(f,U)$ with the compactification given by the torus instead of $U^\infty=X_2$ (with $X_2$ the second space in Figure~\ref{toro colapso circunferencia}) and using \cite{B-B-P-T}, it is easy to obtain $\varLambda_{\mathrm{comb}}(f,U)=2\geq1$. Hence, if Conjecture~\ref{conjetura} is true, Theorem~\ref{pseudoteorema} implies the existence of a fixed point of $f$.
\end{example}
     \begin{figure}[htb]
    \centering
     \includegraphics[scale=0.45]{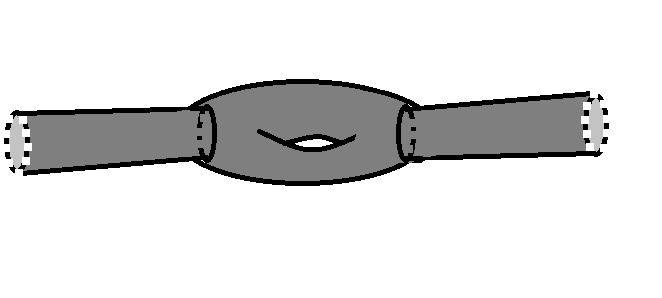} 
     \caption{Space $U$.}
     \label{toro infinito doble}
\end{figure}

Note that the number of cases where we could apply Theorem~\ref{pseudoteorema} and the computational advantage of the combinatorial Lefschetz number would be now even much greater than before. 
\section*{Ethical Statement}
On behalf of all authors, the corresponding author states that there is no conflict of interest.

 \end{document}